\documentclass[11pt,draft,reqno]{amsart}

\makeatletter
\def\section{\@startsection{section}{1}%
	\z@{1\linespacing\@plus\linespacing}{1\linespacing}
	{\bfseries
		\centering}}
\def\@secnumfont{\bfseries}
\makeatother
\usepackage[margin=1.0in, letterpaper]{geometry}

\usepackage[parfill]{parskip}    

\makeatletter
\def\subsubsection{\@startsection{subsubsection}{3}%
	\z@{.5\linespacing\@plus.7\linespacing}{.1\linespacing}%
	{\normalfont\itshape}}
\makeatother
\makeatletter
\def\subsection{\@startsection{subsection}{3}%
	\z@{.5\linespacing\@plus.7\linespacing}{.1\linespacing}%
	{\normalfont\itshape}}
\makeatother

\usepackage{amsthm, amsmath, amssymb, amsfonts, mathrsfs, hyperref, url, cancel}

\usepackage[T1]{fontenc}
\usepackage[usenames]{color}
\definecolor{Green}{rgb}{0,0.5,0.0}

\newtheorem{theorem}{Theorem}

\newtheorem{lemma}[theorem]{Lemma}

\theoremstyle{definition}
\newtheorem{definition}[theorem]{Definition}

\theoremstyle{remark}
\newtheorem{remark}[theorem]{Remark}

\DeclareFontFamily{U}{matha}{\hyphenchar\font45}
\DeclareFontShape{U}{matha}{m}{n}{
	<5> <6> <7> <8> <9> <10> gen * matha
	<10.95> matha10 <12> <14.4> <17.28> <20.74> <24.88> matha12}{}
\DeclareSymbolFont{matha}{U}{matha}{m}{n}
\DeclareFontSubstitution{U}{matha}{m}{n}

\DeclareFontFamily{U}{mathx}{\hyphenchar\font45}
\DeclareFontShape{U}{mathx}{m}{n}{
	<5> <6> <7> <8> <9> <10>
	<10.95> <12> <14.4> <17.28> <20.74> <24.88>
	mathx10}{}
\DeclareSymbolFont{mathx}{U}{mathx}{m}{n}
\DeclareFontSubstitution{U}{mathx}{m}{n}

\DeclareMathDelimiter{\vvvert}{0}{matha}{"7E}{mathx}{"17}

\makeatletter
\DeclareFontFamily{OMX}{MnSymbolE}{}
\DeclareSymbolFont{MnLargeSymbols}{OMX}{MnSymbolE}{m}{n}
\SetSymbolFont{MnLargeSymbols}{bold}{OMX}{MnSymbolE}{b}{n}
\DeclareFontShape{OMX}{MnSymbolE}{m}{n}{
	<-6>  MnSymbolE5
	<6-7>  MnSymbolE6
	<7-8>  MnSymbolE7
	<8-9>  MnSymbolE8
	<9-10> MnSymbolE9
	<10-12> MnSymbolE10
	<12->   MnSymbolE12}{}
\DeclareFontShape{OMX}{MnSymbolE}{b}{n}{
	<-6>  MnSymbolE-Bold5
	<6-7>  MnSymbolE-Bold6
	<7-8>  MnSymbolE-Bold7
	<8-9>  MnSymbolE-Bold8
	<9-10> MnSymbolE-Bold9
	<10-12> MnSymbolE-Bold10
	<12->   MnSymbolE-Bold12}{}

\let\llangle\@undefined
\let\rrangle\@undefined
\DeclareMathDelimiter{\llangle}{\mathopen}%
{MnLargeSymbols}{'164}{MnLargeSymbols}{'164}
\DeclareMathDelimiter{\rrangle}{\mathclose}%
{MnLargeSymbols}{'171}{MnLargeSymbols}{'171}
\makeatother
\usepackage{soul}
\usepackage{changes}
\definechangesauthor[name={RS}, color=red]{RS}

\begin{document}
	
	\title[Wick SHE]
	{A Feynman-Kac approach for the spatial derivative of the solution to the Wick stochastic heat equation driven by time homogeneous white noise}
	
	\author{Hyun-Jung Kim}
	\curraddr[H.-J. Kim]{Department of Mathematics, UCSB\\
		Santa Barbara, CA 93106, USA}
	\email[H.-J. Kim]{hjkim@ucsb.edu}
	\urladdr{https://sites.google.com/view/hyun-jungkim}
	
	\author{Ramiro Scorolli}
	\curraddr[R. Scorolli]{Dipartimento di Scienze Statistiche Paolo Fortunati, Università di Bologna, Bologna, Italy.}
	\email[R. Scorolli]{ramiro.scorolli2@unibo.it}
	
	\date{\today}
	


\begin{abstract}
	We consider the (unique) mild solution $u(t,x)$ of a 1-dimensional stochastic heat equation  on $[0,T]\times\mathbb R$ driven by time-homogeneous white noise in the Wick-Skorokhod sense. The main result of this paper is the computation of the spatial derivative of $u(t,x)$, denoted by $\partial_x u(t,x)$, and its representation as a Feynman-Kac type closed form. The chaos expansion of $\partial_x u(t,x)$ makes it possible to find its (optimal) H\"older regularity especially in space. 
\end{abstract}

\maketitle


\section{Introduction}

As an extension of the paper \cite{S2021}, we will further investigate the (unique) mild solution of the 1-dimensional stochastic heat equation (SHE):
\begin{equation}\label{eq:1dSHE}
	\begin{cases}
		\partial_t u(t,x)=\displaystyle\frac{1}{2}\partial^2_{xx} u(t,x)+u(t,x)\diamond \dot{W}(x), \quad t\in (0,T],\ x\in \mathbb{R}, \\
		u(0,x)=u_0(x),\quad x\in \mathbb{R},
	\end{cases}
\end{equation}
where $T>0$, $u_0$ is a function satisfying certain conditions, $\dot{W}(x)$ is a space-only Gaussian \textit{white noise} on a complete probability space $\left(\Omega,\mathcal{F},\mathbb{P}^W\right)$, and $\diamond$ stands for the \textit{Wick product}. In other words, the corresponding stochastic integration in \eqref{eq:1dSHE} is interpreted in the Wick-Skorokhod sense. We postpone all technical definitions to the following section.

	\begin{definition}\label{def:mildsol}
		We say $u:[0,T]\times \mathbb{R}\times \Omega \to \mathbb{R}$ is said to be a \textit{mild solution} of \eqref{eq:1dSHE} if for any fixed $(t,x)\in [0,T]\times \mathbb{R}$, $u(t,x)\in L^2\left(\mathbb{P}^W\right)$ and it satisfies
		\begin{align}\label{eq:mildsol}
		u(t,x)= \int_{\mathbb{R}}p(t,x-y)u_0(y)dy+ \int_0^t\int_{\mathbb{R}} p(t-s,x-y)u(s,y)\diamond \dot W(y)dyds,\ \mbox{$\mathbb{P}^W$-almost surely},
		\end{align}
		where $p(t,x):=\displaystyle\frac{1}{\sqrt{2\pi t}}e^{-\frac{x^2}{2t}}$ is the Gaussian heat kernel.
	\end{definition}
Let $\left\{u(t,x)\right\}_{(t,x)\in [0,T]\times \mathbb{R}}$ be a mild solution of \eqref{eq:1dSHE}. Then for
	any fixed $(t,x)$, the random variable $u(t,x)$ admits the following multiple Wiener chaos expansion (e.g. \cite{H2002}, \cite{HHNT2015} or \cite{U1996}):
	\begin{align}\label{eq:MWsol}
	u(t,x)=\sum_{n=0}^{\infty}I_n\left(F_n^{\mathtt{MW}}(t,x)\right),
	\end{align}
	where $I_n$ is the $n$-th multiple Wiener integral with respect to $W$,
	\begin{align*}
	&F_0^{\mathtt{MW}}(t,x)=\displaystyle\int_{\mathbb{R}}p(t,x-y)u_0(y)\;dy;\\
	&F_n^{\mathtt{MW}}(t,x;y_1,\dots,y_n)\\
	&=\frac{1}{n!}\int_{[0,t]^n} p\left(t-r_{\rho(n)},x-y_{\rho(n)}\right)\times \cdots
	\times  p\left(r_{\rho(2)}-r_{\rho(1)},y_{\rho(2)}-y_{\rho(1)}\right)F_0^{\mathtt{MW}}(r_{\rho(1)},x_{\rho(1)})\; d\mathbf{r},\ n\geq 1,
	\end{align*}
	and $\rho$ denotes the permutation of $\{1,\dots,n\}$ such that $0<r_{\rho(1)}<\cdots<r_{\rho(n)}<t$. For simplicity, we have denoted $d\mathbf{r}:=dr_1dr_2\cdots dr_n$.
	
	To distinguish among different representations of the mild solution, let us call \eqref{eq:MWsol} the \textbf{\textit{multiple Wiener solution}} $u^{\mathtt{MW}}(t,x)$ of \eqref{eq:1dSHE}. There are a few papers considering this representation:
	\begin{itemize}
		\item[(i)] The paper \cite[Theorem 3.1]{U1996} shows that $u^{\mathtt{MW}}$ is the unique mild solution in 
		$C\left([0,T];L^2(\mathbb{R});L^2(\Omega)\right)$ if $u_0\in L^2(\mathbb{R})$ by showing,
		$$\sup_{t\in [0,T]}\sum_{n=0}^{\infty}n!\int_{\mathbb{R}}\left\|F^{\mathtt{MW}}_n(t,x;\bullet)\right\|_{L^2(\mathbb{R}^{n})}^2dx\leq C\|u_0\|_{L^{2}(\mathbb{R})}^2<\infty.$$
		Note that $u_0\in L^2(\mathbb{R})$ does not cover $u_0\equiv 1$.
		\item[(ii)] When $u_0\in L^{\infty}(\mathbb{R})$, \cite[Section 4]{H2002} shows that
			$$\sup_{(t,x)\in [0,T]\times \mathbb{R}}\sum_{n=0}^{\infty}n!\left\|F^{\mathtt{MW}}_n(t,x;\bullet)\right\|_{L^2(\mathbb{R}^n)}^2\leq C\|u_0\|_{L^{\infty}(\mathbb{R})}^2<\infty.$$
			Hence, we can say that $u^{\mathtt{MW}}$ is the unique mild solution in 
			$C\left([0,T]\times\mathbb{R};L^2(\Omega)\right)$ if $u_0\in L^{\infty}(\mathbb{R})$.
	\end{itemize}
	
	There is an alternative chaos expansion of the mild solution $u$ (e.g. \cite[Theorem 3.11]{LR2009}):
		\begin{align}\label{eq:CSsol}
	u(t,x)=\sum_{\alpha\in \mathcal{J}} u^{\mathtt{CS}}_{\alpha}(t,x)\xi_{\alpha},
	\end{align}
	where $\xi_{\alpha}$ and $\mathcal{J}$ are defined in Section \ref{sec:C-Hermite}. Letting $\mathbb T^n_{[0,t]}:=\{0\leq s_1\leq\cdots\leq s_n\leq t\}$, we can write 
	$u^{\mathtt{CS}}_{(\mathbf{0})}(t,x)=\displaystyle\int_{\mathbb{R}}p(t,x-y)u_0(y)dy$, and for $|\alpha|=n\geq 1$,
	\begin{align*}
	u^{\mathtt{CS}}_{\alpha}(t,x)=\sqrt{n!}\int_{\mathbb T^n_{[0,t]}}\int_{\mathbb R^n} p(t-s_n,x-y_n)\cdots p(s_2-s_1,y_2-y_1)u^{\mathtt{CS}}_{(\mathbf{0})}(s_1,y_1)\; \mathfrak{e}_{\alpha}(y_1,\dots,y_n)\;d\mathbf{s} d\mathbf{y},
	\end{align*} 
	for	$\alpha\in \mathcal J_n:=\{\alpha\in\mathcal J:|\alpha|=n\}$,
	and $\{\mathfrak{e}_{\alpha},\alpha\in \mathcal{J}_n\}$ forms an orthonormal basis of the symmetric part of $L^2(\mathbb R^n)$. We will call \eqref{eq:CSsol} the \textbf{\textit{chaos solution}} $u^{\mathtt{CS}}(t,x)$ of \eqref{eq:1dSHE}. The existence and uniqueness of this representation can be proved by showing the following:
	\begin{itemize}
		\item[(i')] We can prove that $u^{\mathtt{CS}}$ is the unique mild solution in $C\left([0,T];L^2(\mathbb{R});L^2(\Omega)\right)$ when $u_0\in L^2(\mathbb{R})$ by showing (c.f. \cite[Theorem 4.1]{KL2017})
		\begin{align}\label{eq:CSbdd-1}
		\sup_{t\in[0,T]} \sum_{n=0}^{\infty} \sum_{\alpha\in \mathcal{J}_n} \left\|u_{\alpha}^{\mathtt{CS}}(t,\bullet)\right\|_{L^2(\mathbb{R})}^2\leq C \|u_0\|_{L^2(\mathbb{R})}^2<\infty.
		\end{align}
		\item[(ii')] We can also show that $u^{\mathtt{CS}}$ is the unique mild solution in $C\left([0,T]\times\mathbb{R};L^2(\Omega)\right)$ if $u_0\in L^{\infty}(\mathbb{R})$ (c.f. \cite[Theorem 4.3]{KL2017}) by achieving 
		\begin{align}\label{eq:CSbdd-2}
		\sup_{(t,x)\in[0,T]\times \mathbb{R}} \sum_{n=0}^{\infty} \sum_{\alpha\in \mathcal{J}_n} \left|u_{\alpha}^{\mathtt{CS}}(t,x)\right|^2\leq C \|u_0\|_{L^{\infty}(\mathbb{R})}^2<\infty.
		\end{align}
	\end{itemize}
Indeed, \eqref{eq:CSbdd-1} and \eqref{eq:CSbdd-2} can be easily obtained as follows:
\begin{itemize}
	\item[\eqref{eq:CSbdd-1}] To use the same argument as \cite[Theorem 4.1]{KL2017}, it is enough to show $$U_0:=\int_{\mathbb{R}} \left(\int_{\mathbb{R}} p(s,y-z_1)u_0(z_1)dz_1\right)\cdot \left(\int_{\mathbb{R}} p(s,y-z_2)u_0(z_2)dz_2\right) dy\leq \|u_0\|^2_{L^2(\mathbb{R})},$$ and it is clear by semigroup property and H\"older inequality,
	\begin{align*}
	U_0&=\int_{\mathbb{R}}\int_{\mathbb{R}} p(s+r,z_1-z_2) u_0(z_1)u_0(z_2)dz_1dz_2\\
	&=\int_{\mathbb{R}} p(s+r,z_1)\int_{\mathbb{R}} u_0(z_1+z_2)u_0(z_2)dz_2dz_1\leq \|u_0\|^2_{L^2(\mathbb{R})}.
	\end{align*}
	\item[\eqref{eq:CSbdd-2}] To use the same argument as \cite[Theorem 4.3]{KL2017}, it is enough to show $\left|\int_{\mathbb{R}}p(t,x-y)u_0(y)dy\right|\leq \|u_0\|_{L^{\infty}(\mathbb{R})}$, and it automatically follows from the fact $\int_{\mathbb{R}}p(t,x)dx=1$.	
\end{itemize}

Then, it is not surprising that $u^{\mathtt{MW}}=u^{\mathtt{CS}}$ if $u_0\in L^{\infty}(\mathbb{R})$ since the mild solution is unique in $C\left([0,T]\times\mathbb{R};L^2(\Omega)\right)$.
	
We now discuss one more possible representation of the mild solution. In fact, the condition \eqref{eq:mildsol} is equivalent to the following because $\mathbb{D}^{1,2}$ is dense in $L^2\left(\mathbb{P}^W\right)$. (e.g. \cite[Exercise 1.1.7 and Corollary 1.5.1]{N2006}): For any random variable $F\in \mathbb{D}^{1,2}$,
\begin{align}\label{eq:mildsol-1}
\mathbb{E}\left[F\cdot u(t,x)\right]=\mathbb{E}[F]\cdot \int_{\mathbb{R}}p(t,x-y)u_0(y)dy+ \mathbb{E}\left[F\cdot \int_0^t\int_{\mathbb{R}} p(t-s,x-y)u(s,y)\diamond \dot W(y)dyds\right].
\end{align}
Here, Malliavin derivative $D$ and the Sobolev-Malliavin space $\mathbb{D}^{1,2}$ are defined in Section \ref{sec:Malliavin}.

Moreover, it is known that \eqref{eq:mildsol-1} is equivalent to
\begin{align}\label{eq:mildsol-2}
\mathbb{E}\left[F\cdot u(t,x)\right]=\mathbb{E}[F]\cdot \int_{\mathbb{R}}p(t,x-y)u_0(y)dy+\mathbb{E}\left[\left\langle \int_0^t p(t-s,x-\bullet)u(s,\bullet)ds,D_{(\bullet)}F \right\rangle_{L^2(\mathbb{R})}\right],
\end{align}
if $\mathbb{E}\left(F\cdot \left(\int_{\mathbb{R}}h(y)\diamond \dot W(y)dy\right)\right)=\mathbb{E}\left(\langle DF, h\rangle_{L^2(\mathbb{R})}\right)$ for $h\in L^2\left(\mathbb{P}^W;L^2(\mathbb{R})\right)$ and $\int_{\mathbb{R}} h(y)\diamond \dot W(y)dy\in L^2\left(\mathbb{P}^W\right)$ (e.g. \cite[Section 2.5]{NP2012}).

Using \eqref{eq:mildsol-2} and a Wong-Zakai-type approximation, the paper \cite{S2021} gives a Feynman-Kac representation of the unique mild solution of \eqref{eq:1dSHE} when $u_0\in L^{\infty}(\mathbb{R})$. This is given by
\begin{align*}
u(t,x)=\mathbb E^B\left[u_0(B_t^x)\exp\{\Psi_{t,x}\}\right],
\end{align*}
where $\{B_t^x\}_{t\geq 0}$ is a one-dimensional Brownian motion starting at $x$, and for fixed $(t,x)\in [0,T]\times\mathbb R$, the random variable $\Psi_{t,x}$ is given by 
\begin{align*}
\Psi_{t,x}:=\int_{\mathbb R} L_y^x(t)dW(y)-\frac{1}{2}\int_{\mathbb R}|L_y^x(t)|^2 dy.
\end{align*}
Here $L_a^x(t)$ denotes the local time of $\left\{B_s^x\right\}_{s\geq 0}$ at level $a$ and time $t$. Let us call the Feynman-Kac representation the \textbf{\textit{Feynman-Kac solution}} $u^{\mathtt{FK}}(t,x)$ of \eqref{eq:1dSHE}.

Combining all, as long as $u_0\in L^{\infty}(\mathbb{R})$, we can say
\begin{align}\label{eq:equivSol}
u:=u^{\mathtt{FK}}=u^{\mathtt{MW}}=u^{\mathtt{CS}}\in C\left([0,T]\times\mathbb{R};L^2(\Omega)\right).
\end{align}
In this paper,
we will provide an alternative proof for the equivalence \eqref{eq:equivSol} using a more direct approach.

The main motivation for the current article is as follows: As we stated above,
the equation \eqref{eq:1dSHE} may have three possible representations for the unique mild solution $u$, namely (I) Feynman-Kac solution $u^{\mathtt{FK}}$, (II) multiple Wiener-It\^o integral solution $u^{\mathtt{MW}}$, and (III) chaos solution $u^{\mathtt{CS}}$. Unfortunately, there is no enough discussion on  H\"older regularity of the mild solution. In particular,
\begin{itemize}
	\item[(I)] There is no H\"older regularity result for $u^{\mathtt{FK}}$ in the existing literature. 
	\item[(II)] For $u^{\mathtt{MW}}$, \cite[Theorem 4.1]{U1996} proves that $u^{\mathtt{MW}}\in \mathcal{C}^{1/2-\varepsilon,1/2-\varepsilon}([0,T],\mathbb{R})$ for any small $\varepsilon>0$ if $u_0\in C^1_b(\mathbb{R})\cap L^2(\mathbb{R})$. Here, $C^1_b(\mathbb{R})$ denotes the space of all bounded differentiable functions on $\mathbb{R}$ with bounded continuous derivatives.
	\item[(III)] On the one hand, no one discusses the regularity of $u^{\mathtt{CS}}$ on the whole line. On the other hand, the paper \cite{KL2017} discuss the same equation as \eqref{eq:1dSHE}, but on a bounded domain, say $[0,\pi]$; the authors show that there exists a unique mild solution (using chaos expansion) $u_b^{\mathtt{CS}}\in C([0,T];L^2([0,\pi]);L^2(\Omega))$ if $u_0\in L^2([0,\pi])$, and moreover $u_b^{\mathtt{CS}}\in \mathcal{C}^{3/4-\varepsilon,3/2-\varepsilon}([0,T]\times [0,\pi])$ for any small $\varepsilon>0$ if $u_0\in \mathcal{C}^{3/2}([0,\pi])$.
\end{itemize}
Since H\"older continuity is a local property, it is natural to expect that $u^{\mathtt{CS}}\in \mathcal{C}^{3/4-\varepsilon,3/2-\varepsilon}([0,T]\times \mathbb{R})$ for any small $\varepsilon>0$ (under a suitable initial condition on $u_0$) like the bounded case $u_b^{\mathtt{CS}}$. Furthermore, it is impossible that the other representations $u^{\mathtt{FK}}$ and $u^{\mathtt{MW}}$ have a different regularity from the one of $u^{\mathtt{CS}}$ (by uniqueness). 
In this sense, we would say that the existing H\"older regularity results of the mild solution on $\mathbb{R}$ should be improved, and in this paper, we will suggest an idea of how to get the desired result. We emphasize that the regularity almost $3/4$ in time and almost $3/2$ in space is optimal in the classical PDE sense, since $\dot W$ is understood to have regularity $-1/2-\varepsilon$ for any small $\varepsilon>0$ (c.f. \cite[Lemma 1.1]{HL2015}).

The main feature of this paper is to find the optimal spatial regularity of the unique mild solution $u$. The first task is to find $\partial_x u$ and check if it is well-defined. One can compute $\partial_x u$ from $u^{\mathtt{CS}}$ using the same argument as \cite{KL2017}, but we will focus on the Feynman-Kac representation and compute the spatial derivative of $u$ using $u^{\mathtt{FK}}$ as a main part of this paper. This approach allows us to obtain a Feynman-Kac-type closed formula for $\partial_x u$. 
We remark that we can also derive
the chaos decomposition of $\partial_x u$ using $u^{\mathtt{FK}}$, and it is exactly the same as the one after differentiating $u^{\mathtt{CS}}$ with respect to $x$ directly.
With this in hand, we can achieve the optimal H\"older regularity of $\partial_x u$ that is almost $1/4$ in time and almost $1/2$ in space.


\section{Preliminaries}
\subsection{Elements of white noise analysis}

In this section, we will give a brief outline of the white noise tools which will be used in this article. The interested reader for more details is referred to \cite{HKSP1990} and \cite{K2018}.

For $n\in \mathbb{N}$, let $\mathscr{S}(\mathbb R^n)$ be the Schwartz space of rapidly decreasing functions on $\mathbb R^n$, and $\mathscr{S}'(\mathbb R^n)$ the dual of $\mathscr{S}(\mathbb R^n)$, which is called the space of tempered distributions. We further define the white noise probability space by $(\mathscr{S}'(\mathbb R),\mathcal B,\mu)$, where $\mathcal B$ is the Borel $\sigma$-algebra on $\mathscr{S}'(\mathbb R)$, i.e. the $\sigma$-algebra generated by the \textit{cylindrical sets}, and $\mu$ is the \textit{standard Gaussian measure in $\mathscr{S}'(\mathbb R)$} (see \cite[Section 3.1]{K2018} for more details). Specifically, the measure $\mu$ satisfies
\begin{align*}
\int_{\mathscr{S}'(\mathbb R)}e^{i\langle\omega,\varphi\rangle}d\mu(\omega)=e^{-\frac{1}{2}|\varphi|_0^2},\quad\quad\varphi\in\mathscr{S}(\mathbb R),
\end{align*}
where $|\bullet|_0$ denotes the norm in $L^2(\mathbb R)$.

Let $W$ denote the canonical coordinate process (or Wiener integral) on $\mathscr{S}'(\mathbb R)$ given by 
$$W_{\phi}=\langle \omega,\phi\rangle,\; \phi\in \mathscr{S}(\mathbb R),\; \omega\in \mathscr{S}'(\mathbb R).$$ Here, $\langle \bullet,\bullet \rangle$ denotes the dual pairing between $\mathscr{S}(\mathbb R)$ and $\mathscr{S}'(\mathbb R)$.

Note that we can extend $W$ continuously in $L^2(\mu)$ to $L^2(\mathbb R)$ as
\begin{align*}
	\langle\omega,\phi\rangle:=L^2(\mu)-\lim_{k\to\infty}\langle \omega,\phi_k\rangle,\quad \phi\in L^2(\mathbb R),
\end{align*}
where $\{\phi_k\}_{k\in\mathbb N}$ is any sequence in $\mathscr{S}(\mathbb R)$ such that $\phi_k\to\phi$ in $L^2(\mathbb R)$.
In particular, if we define 
\begin{align}\label{eq:Bm}
	W(x)(\omega):=
	\begin{cases}
		\langle \omega, \chi_{[0,x]}\rangle,\quad \mbox{if $x\geq 0,\ \omega\in \mathscr{S}'(\mathbb R)$};\\
		-\langle \omega, \chi_{[x,0]}\rangle,\quad \mbox{if $x<0,\ \omega\in \mathscr{S}'(\mathbb R)$,}
	\end{cases}
\end{align}
it is easy to check that $\{W(x)\}_{x\in\mathbb R}$ is a Brownian motion.

\subsubsection{Chaos expansion in terms of multiple Wiener integrals}\label{sec:C-MW}

It is well-known that any random variable $F\in L^2(\mu)$ can be written as
\begin{align*}
	F=\sum_{n=0}^{\infty} I_n(f_n),
\end{align*}
where $I_n$ is the multiple Wiener integral of order $n$ with respect to the Brownian motion in \eqref{eq:Bm}, and $f_n$ is a symmetric element of $L^2(\mathbb R^n)$. This is the Wiener-It\^o-Segal isomorphism between square integrable \textit{Brownian functionals} and the the symmetric Fock space \cite{J1997}, i.e.
\begin{align*}
	L^2(\mu)\cong \bigoplus_{n=0}^{\infty} \operatorname{Sym}L^2(\mathbb R)^{\otimes n}.
\end{align*} 
It turns out that the following isometry property holds
\begin{align*}
	\|F\|^2_2=\sum_{n=0}^{\infty}n!|f_n|_{0,n}^2,
\end{align*}
where $\|\bullet\|_2$ denotes the norm in $L^2(\mu)$ and $|\bullet|_{0,n}$ denotes the norm in $L^2(\mathbb R^n), n\in\mathbb N$, whenever $n=1$ we shall omit the second sub-index.

\subsubsection{\textbf{Chaos expansion in terms of Hermite polynomials}}\label{sec:C-Hermite}

Let $\{e_j\}_{j\in\mathbb N}\subset \mathscr{S}(\mathbb R)$ be the family of \textit{Hermite functions} defined by 
\begin{align*}
	e_j(x):=(-1)^{j-1}\left(\sqrt{\pi}2^{j-1} (j-1)!\right)^{-1/2}e^{x^2/2}\frac{d^{(j-1)}}{dx^{(j-1)}}e^{-x^2},\ j\in\mathbb N,
\end{align*}
{where $\displaystyle\frac{d^0}{dx^0}$ is the identity operator.}
It is known that $\{e_j\}_{j\in\mathbb N}$ forms a complete orthonormal basis (CONB) of $L^2(\mathbb R)$.
We next let $\mathcal J:=(\mathbb N_0^{\mathbb N})_c$ be the collection of multi-indices $\alpha=(\alpha_1,\alpha_2,\dots)$ such that every $\alpha_j$ is a non-negative integer and there are only finitely many non-zero components. In this case, we define a few notations: 
\begin{itemize}
	\item $(\mathbf{0})$ is the multi-index with all zeroes;
	\item $\alpha^-_{(j)}:=
	\left(\alpha_1,\alpha_2, \dots,\mbox{max}(\alpha_j-1,0),\alpha_{j+1},\dots\right)\in \mathcal{J}$;
	\item $\alpha^{+}_{(j)}:=(\alpha_1,\alpha_2, \dots,\alpha_j+1,\alpha_{j+1},\dots)\in \mathcal{J};$
	\item $|\alpha|:=\sum_{j=1}^{\infty} \alpha_j$;
	\item $\alpha!:=\prod_{j=1}^{\infty} \alpha_j!$.
\end{itemize} 
We also define the collection of random variables $\Xi:=\{\xi_{\alpha},\alpha\in\mathcal J\}$ by
\begin{align*}
	\xi_{\alpha}:=\prod_{j=1}^{\infty} \left(\frac{H_{\alpha_j}(W_{e_j})}{\sqrt{\alpha_j!}}\right),
\end{align*}
where 
\begin{align*}
	H_n(x)=(-1)^n e^{x^2/2}\frac{d^n}{dx^n}e^{-x^2/2}
\end{align*}
is the Hermite polynomial of order $n$.

One of the famous Cameron and Martin's theorems \cite{CM1947} states that the family $\Xi$ forms an orthonormal basis in $L^2(\mu)$, and thus for $F\in L^2(\mu)$, we have the following chaos expansion
\begin{align*}
	F=\sum_{\alpha\in\mathcal J} F_{\alpha}\xi_{\alpha}\quad \mbox{and}\quad \|F\|_2^2=\sum_{\alpha\in\mathcal J}F_{\alpha}^2,\quad \mbox{where}\ F_{\alpha}:=\mathbb E[F\xi_{\alpha}].
\end{align*}

In fact, both approaches to the chaos expansion for $F\in L^2(\mu)$ in Sections \ref{sec:C-MW} and \ref{sec:C-Hermite} are equivalent, and we can go from one to the other without particular difficulties. For example, see Section \ref{sec:MW-CS} below.

\subsubsection{\textbf{S-transform, Wick product, and generalized random variables}}

Now we introduce one of the main tools we will employ in our proofs, namely the S-transform. Let $F\in L^2 (\mu)$ and $\phi\in L^2(\mathbb R)$. Then the S-transform of $F$ is given by
\begin{align*}
	S(F)(\phi):=\mathbb E\left[F\times\mathcal E(\phi)\right],\quad \mbox{where}\quad \mathcal E(\phi):= \exp\bigg\{W_{\phi}-\frac{1}{2}|\phi|_0^2\bigg\}.
\end{align*}
Note that $\mathcal E(\phi)$ 
is called \textit{stochastic exponential} or \textit{Wick exponential} (e.g. \cite{HOUZ2010}).

In certain applications, one may be interested in spaces that are larger than $L^2(\mu)$, and one is naturally led to consider spaces of \textit{generalized} random variables (see for instance \cite{HKSP1990}, \cite{H2016}, \cite{K2018}). In this article, we will solely consider the space of \textit{Hida distributions}, and will briefly introduce them in the following.

Let $A$ be the operator given by $A=-\displaystyle\frac{d^2}{dx^2}+x^2+1$, and for $F\in L^2(\mu)$ with $F=\displaystyle\sum_{n=0}^{\infty} I_n(f_n)$ satisfying
\begin{align*}
	\sum_{n=0}^{\infty} n!|A^{\otimes n}f_n|_{0,n}^2<\infty,
\end{align*}
we define $\Gamma(A)F\in L^2(\mu)$ by
\begin{align*}
	\Gamma(A)F:=\sum_{n=0}^{\infty} I_n(A^{\otimes n}f_n).
\end{align*}
Sometimes $\Gamma(A)$ is referred to as the \textit{second quantization} of the operator $A$. 

\begin{remark}
It is known (e.g. \cite[Theorem 2]{S1971}) that if we define $|\bullet|_p:=|A^{p}\bullet|_0$, the family of \textit{p-seminorms} $\left\{|\bullet|_p,\; p\geq 0\right\}$ are equivalent to the usual seminorms in the Schwartz space $\mathscr{S}(\mathbb R)$, i.e. they generate the same topology.
\end{remark}

\begin{definition}
	Let $(S_p)^*$ be the completion of $L^2(\mu)$ with respect to the norm $$\|\bullet\|_{-p}:=\left\|\Gamma(A^{-p})\bullet\right\|_2,\; p>0.$$
	Then the space of \textit{Hida distributions} is given by
	\begin{align*}
		(S)^*:=\bigcup_{p\geq 0} (S_p)^*.
	\end{align*}
\end{definition}
Any element $\Phi\in (S)^*$ can be represented as the formal series:
\begin{align*}
	\Phi=\sum_{n=0}^{\infty}I_n(F_n),\quad F_n\in \mathscr{S}'_{sym}(\mathbb R^n):=\operatorname{Sym} \mathscr{S}'(\mathbb R)^{\otimes n}
\end{align*}
such that  
\begin{align*}
	\sum_{n=0}^{\infty}n!|(A^{-p})^{\otimes n}F_n|_{0,n}^2<\infty.
\end{align*}

In this setting, we can give a proper meaning to the \textit{white noise} $\partial_{x}W(x)$ or $\dot W(x)$ as a Hida distribution and in an slight abuse of notation, we will denote
\begin{align*}
	\dot W(x;\omega):=\langle \omega,\delta_x\rangle
\end{align*}

Here, $\delta_x$ stands for the Dirac-delta function.
We note that \cite[Section 3.4]{K2018} that $\dot W\in (S_p)^*$ for any $p>5/12$.

\begin{remark}
	The S-transform can be extended naturally to the space of Hida distributions by a duality argument, i.e. for $\Phi\in(S)^*$, we define the S-transform of $\Phi$ as
	\begin{align*}
		S(\Phi)(\phi):=\llangle \Phi,\mathcal E(\phi)\rrangle,\quad \phi\in \mathscr{S}(\mathbb R),
	\end{align*}
	where $\llangle\bullet,\bullet\rrangle$ stands for the bilinear dual paring between the space of Hida distributions $(S)^*$ and its dual space, denoted by $(S)$. Note that $(S)$ is called the space of Hida test functions (see \cite{K2018} for further details).
\end{remark}

\begin{definition}\label{U-functional}
	A function $F:\mathscr{S}(\mathbb R)\to\mathbb C$ is called a U-functional if
	\begin{enumerate}
		\item For every $\phi,\varphi\in\mathscr{S}(\mathbb R)$ the mapping $\mathbb R\ni\lambda\mapsto F(\lambda\phi+\varphi)\in\mathbb C$ has an entire extension to $z\in\mathbb C$.
		\item There are constants $0<K_1,K_2,p<\infty$ such that 
		$$|F(\varphi)|\leq K_1\exp\left(K_2|\varphi|_p^2\right),\quad \varphi\in \mathscr{S}(\mathbb R).$$
	\end{enumerate}
\end{definition}
We are now ready to introduce a characterization result for the space $(S)^*$.
\begin{theorem}\cite[theorem 1.2]{PS1991}\label{thm:StransformBijection}
	The S-transform defines a bijection between the space $(S)^*$ and the space of U-functionals.
\end{theorem}

It is well-known that we cannot in general define the product between generalized functions, and this impossibility obviously extends also to generalized random variables. However, we are still able to define a particular type of \textit{renormalized} product called \textit{Wick product} and denoted by $\diamond$ in the following way.

\begin{definition}\label{def: Wick product}\cite[Definition 8.11]{K2018}
	Let $\Phi,\Psi\in (S)^*$ be two Hida distributions. We define the \textit{Wick product} between the two elements, denoted by $\Phi\diamond\Psi$, to be the unique Hida distribution satisfying
	\begin{align*}
		S(\Phi\diamond\Psi)(\phi)=S(\Phi)(\phi)\cdot S(\Psi)(\phi)
	\end{align*}
	for every $\phi\in \mathscr{S}(\mathbb R)$.
\end{definition}
\begin{remark}
The Hida space $(S)^*$ is an algebra with respect to the \textit{Wick product}.
\end{remark}
\begin{remark}
	We can also give an alternative characterization of the \textit{Wick product} between generalized random variables in terms of chaos decompositions (e.g.\cite[Corollary 4.22]{HKSP1990}), namely if 
	$\Phi, \Psi\in (S)^*$ are given by the formal series
	\begin{align*}
		\Phi=\sum_{n=0}^{\infty}I_n(F_n),\quad \Psi=\sum_{n=0}^{\infty}I_n(G_n),
	\end{align*}
	then the Wick product between $\Phi, \Psi\in (S)^*$ is defined by the following chaos decomposition:
	\begin{align*}
		\Phi\diamond\Psi=\sum_{n=0}^{\infty}I_n(H_n),
	\end{align*}
	where $H_n=\displaystyle\sum_{j=0}^{n}F_{n-j}\widehat{\otimes} G_{j}$, and $\widehat{\otimes}$ denotes the symmetric tensor product.
\end{remark}

One of the most striking properties of the \textit{Wick product} is its relation with stochastic integration of Skorokhod-It\^o type. In particular if $\{Y(x)\}_{x\in\mathbb R}$ is a Skorokhod integrable process, then we have that 
$$\int_{\mathbb R} Y(x)\delta W(x)=\int_{\mathbb{R}} Y(x)\diamond \dot W(x)dx$$
where the right hand side must be understood as a Pettis integral in $(S)^*$ (see for instance \cite{HOUZ2010} or \cite[Section 13.3]{K2018}) and the left hand side is a Skorokhod integral (see \cite{N2006}). This is the reason why in \eqref{eq:1dSHE} we introduce the \textit{Wick product} $\diamond$ and say that the corresponding stochastic integral should be interpreted in the Skorokhod-It\^o sense.

\subsection{Elements of Malliavin calculus}\label{sec:Malliavin}

For the purpose of this article, we will need a few definitions regarding  Malliavin calculus. The interested reader is referred to \cite{N2006} and \cite{S2005} for a compressive exposition of Malliavin calculus and to \cite{DNOP2009} for the particular case in which the underlying probability space is the white noise probability space.

Let $\mathcal S$ denote the class of smooth random variables $F$ having the form
\begin{align*}
F=f(W_{h_1},\dots,W_{h_n}),
\end{align*}
where $h_1,\dots,h_n\in L^2(\mathbb R)$, and $f$ belong to $C_p^{\infty}(\mathbb R^n)$ which stands for the set of all infinitely continuously differentiable functions such that each function together with all its derivatives has a polynomial growth. We will refer to $\mathcal S$ as the family of \textit{smooth Brownian functionals}.
\begin{definition}
The Malliavin derivative of a smooth Brownian functional $F$ is the $L^2(\mathbb R)$-valued random variable given by
$$DF=\sum_{i=1}^{n}\partial_if(W_{h_1},\dots,W_{h_n})h_i.$$
In the same way, we can define the $k$-th derivative of $F$ for any $k\in\mathbb N$, which will be a $L^2(\mathbb{R})^{\otimes k}$-valued random variable.
\end{definition}

\begin{definition}
Let $\mathcal S$ be the space of smooth Brownian functionals, and define the following semi-norm on $\mathcal{S}$ for $k\geq 1$ and $p\geq 1$,
\begin{align*}
|F|_{k,p}:=\left[\mathbb E\left(|F|^p\right)+\sum_{j=1}^{k}\mathbb E\left(|D^jF|_{0,n}^p\right)\right]^{1/p}.
\end{align*}
We will denote by $\mathbb D^{k,p}$ the completion of the family $\mathcal S$ with respect to the seminorm $|\bullet|_{k,p}$ and for any $F\in \mathbb D^{k,p}$, we will let
$$D^kF=\lim_{n\to\infty} D^kF_n\quad\mbox{in $L^p\left(\mu;L^2(\mathbb R)^{\otimes k}\right)$},$$
where $(F_n)_{n\in\mathbb N}\subset \mathcal S$ is any sequence converging to $F$ in $L^p(\mu)$.
\end{definition}
\begin{remark}
There exists an extension of the Malliavin derivative to an element of the Hida distribution space $(S)^*$ called the \textit{Hida-Malliavin derivative} (see for instance \cite{DNOP2009}).
\end{remark}

Furthermore, we will use the following notation
\begin{align*}
\mathbb D^{\infty,2}:=\bigcap_{k\geq 1}\mathbb D^{k,2},
\end{align*}
and  the following lemma (e.g. \cite{ST1987}).
\begin{lemma}\label{lemma: chaos decomposition kernel}
Let $F\in \mathbb D^{{\infty},2}$ have  the following chaos decomposition $$F=\sum_{n=0}^{\infty}I_n(f_n).$$
Then it holds that $f_n(\bullet)=\displaystyle\frac{1}{n!}\mathbb E[D_{\bullet}^nF]$.
\end{lemma}

Finally we introduce the following space of test functions that was introduced for the first time in \cite{PM1995}.
\begin{definition}\label{def: space G}
For any $\lambda\in\mathbb R$, let  $\mathcal G_{\lambda}$ be the closure of $L^2(\mu)$ with respect to the norm $\left\|\Gamma(e^{\lambda}I)\;\bullet\right\|_{2},$
where $I$ stands for the identity operator.
More explicitly,
\begin{align*}
	\mathcal G_{\lambda}:=\bigg\{F=\sum_{n=0}^{\infty}I_n(f_n)\in L^2(\mu): \sum_{n=0}^{\infty}n!e^{2\lambda n}|f_n|_{0,n}^2<\infty\bigg\},
\end{align*}
and now we set
\begin{align*}
	\mathcal G:=\bigcap_{\lambda\in\mathbb R}\mathcal G_{\lambda}.
\end{align*}
In particular, it is not hard to see the following inclusions:
\begin{align*}
(S)\subset \mathcal G\subset L^2(\mu).
\end{align*}
\end{definition}

We note that if $F\in L^2(\mu)$ can be written as $\displaystyle\sum_{\alpha\in \mathcal{J}}F_{\alpha}\xi_{\alpha}$ (the chaos expansion in terms of Hermite polynomials), then one can show that $F\in \mathcal G_{\lambda}$  if
$$\sum_{n=0}^{\infty}e^{2\lambda n}\sum_{\alpha\in \mathcal{J}_n}|F_{\alpha}|^2<\infty.$$


\subsection{H\"older spaces and classical H\"older regularity results}

In this subsection, we first give a definition of H\"older spaces on $G\subseteq \mathbb{R}$. For $0<\gamma<1$, we let
\begin{align*}
 [f]_{\gamma}:= \sup_{z_1 \neq z_2 \in G} 
	\frac{\left|f(z_1)-f(z_2)\right|}{|z_1-z_2|^{\gamma}}.
\end{align*}

We say that $f$ is H\"older continuous with H\"older exponent $\gamma$
(or H\"older $\gamma$ continuous) on $G$ if
$$
\sup_{z\in G} |f(z)|
+ [f]_{\gamma}<\infty.$$
The collection of H\"older $\gamma$ continuous functions on $G$ is 
denoted by
$\mathcal{C}^{\gamma}\left(G\right)$
with the norm 
$$[[ f ]]_{\gamma}:=\sup_{z\in G}|f(z)|+ [f]_{\gamma}.$$

For $k\in \mathbb{N}$, we say that $f$ is a $k$ times continuously differentiable function on $G$
if the $m$-th derivative of $f$, denoted by $\partial^{m} f$, exists and is continuous for all $m\leq k$.
The collection of $k$ times continuously differentiable functions on $G$ such that $\partial^{k} f \in \mathcal{C}^{\gamma}(G)$ with $0<\gamma<1$, is denoted by $\mathcal{C}^{k+\gamma}(G)$
with the norm
$$[[ f ]]_{k+\gamma}:= 
\sum_{1\leq m\leq k} \sup_{z\in G}|\partial^{m} f(z)|
+[\partial^{k} f]_{\gamma}<\infty.$$

In a similar manner, we can define the H\"older spaces on $[0,T]\times \mathbb{R}$ for $T>0$ as follows. For $0<\gamma_1,\gamma_2<1$, we define
\begin{align*}
	[f]_{\gamma_1,\gamma_2}:=\displaystyle\sup_{(t,x) \neq (s,x) \in [0,T]\times \mathbb{R}} 
	\frac{\left|f(t,x)-f(s,x)\right|}{|t-s|^{\gamma_1}}+\displaystyle\sup_{(t,x) \neq (t,y) \in [0,T]\times \mathbb{R}} 
	\frac{\left|f(t,x)-f(t,y)\right|}{|x-y|^{\gamma_2}}.
\end{align*}
Then, $f$ is said to be H\"older $(\gamma_1,\gamma_2)$ continuous on $[0,T]\times \mathbb{R}$ if $\displaystyle\sup_{(t,x) \in [0,T]\times \mathbb{R}} |f(t,x)|
+ [f]_{\gamma_1,\gamma_2}<\infty$, and
the collection of H\"older $(\gamma_1,\gamma_2)$ continuous functions on $[0,T]\times\mathbb{R}$ is denoted by
$\mathcal{C}^{\gamma_1,\gamma_2}\left([0,T]\times \mathbb{R}\right)$
with the norm 
$$[[ f ]]_{\gamma_1,\gamma_2}:=\sup_{(t,x)\in [0,T]\times \mathbb{R}}|f(t,x)|+ [f]_{\gamma_1,\gamma_2}.$$

Let $k_1,k_2\in \mathbb{N}$ and $0<\gamma_1,\gamma_2<1$. The H\"older space, denoted by $\mathcal{C}^{k_1+\gamma_1,k_2+\gamma_2}([0,T]\times \mathbb{R})$, is defined by the collection of all functions on $([0,T]\times \mathbb{R})$ such that $f$ is $k_1$ times continuously differentiable in $t$ and $k_2$ times continuously differentiable in $x$ and the norm
$$[[ f ]]_{k_1+\gamma_1,k_2+\gamma_2}:= 
\sum_{0\leq i\leq k_1,\ 0\leq j\leq k_2} \sup_{(t,x)\in [0,T]\times \mathbb{R}}\left|\partial^{i}_t\partial^j_x f(t,x)\right|+\left[\partial^{k_1}_t\partial^{k_2}_x f\right]_{\gamma_1,\gamma_2}<\infty.$$
Here, $\partial_t:=\displaystyle\frac{\partial}{\partial t}$ $\left(\partial_x:=\displaystyle\frac{\partial}{\partial x}\right)$ represents the differentiation operator with respect to $t$ (resp. $x$).

Next, we state useful regularity results for the classical solutions of standard homogeneous and inhomogeneous heat equations on $[0,T]\times \mathbb{R}$.

Recall the Gaussian heat kernel $p(t,x)=\displaystyle\frac{1}{\sqrt{2\pi t}}e^{-\frac{x^2}{2t}}$. Let us define
$$(Pf)(t,x):=\displaystyle\int_{\mathbb{R}} p(t,x-y)f(y)dy,\quad (P\star f)(t,x):=\displaystyle\int_0^t \int_{\mathbb{R}} p(t-s,x-y)f(s,y)dyds.$$ 

\begin{lemma}\label{lem:holder}\cite[Chapter IV, Section 2]{LSU1968}
	Let $T>0$, $0<\gamma\notin \mathbb{N}$, and $n,m\in \mathbb{N}_0$. Then,  
	\begin{itemize}
		\item[(1)] If $2n+m=\rho$ and $f\in \mathcal{C}^{\gamma}(\mathbb{R})$,
		$$[[\partial_t^{n}\partial_x^{m} (Pf)]]_{\frac{\gamma-\rho}{2},\gamma-\rho}\leq C[[ f ]]_{\gamma}$$
		for some constant $C>0$.
		\item[(2)] Let $\lfloor \cdot \rfloor$ be the greatest integer function. If $\rho=\lfloor \gamma \rfloor$, $2n+m=\rho+2$, and $f\in \mathcal{C}^{\frac{\gamma}{2},\gamma}([0,T]\times \mathbb{R})$,
		$$[[\partial_t^{n}\partial_x^{m} (P\star f)]]_{\frac{\gamma-\rho}{2},\gamma-\rho}\leq C[[ f ]]_{\frac{\gamma}{2},\gamma},$$
		for some constant $C>0$.
	\end{itemize}	
\end{lemma}

\section{Direct comparisons among $u^{\mathtt{FK}}$, $u^{\mathtt{MK}}$ and $u^{\mathtt{CS}}$}

Let $\left(\Omega,\mathcal F,\mathbb P^W\right)=\left(\mathscr{S}'(\mathbb R),\mathcal B,\mu\right)$ be our main probability space and introduce an auxiliary one $\left(\widetilde{\Omega},\widetilde{\mathcal{F}},\mathbb P^B\right)$ carrying a one-dimensional Brownian motion $\{B_t\}_{t\geq 0}$. Furthermore, let $\mathbb E^W$ ($\mathbb E^B$) denote the expectation with respect to $\mathbb P^W$ (resp. $\mathbb P^B$).

Moreover, we suppose that $u_0\in L^{\infty}(\mathbb{R})$ so that we have by uniqueness
$$u^{\mathtt{FK}}=u^{\mathtt{MW}}=u^{\mathtt{CS}}\in C([0,T]\times \mathbb{R};L^2(\Omega)).$$
The aim of this section is to give an alternative and more direct proof to show $u^{\mathtt{FK}}=u^{\mathtt{MW}}=u^{\mathtt{CS}}$. For the sake of simplicity and consistency, we will denote, for $(t,x)\in [0,T]\times \mathbb{R}$,
$$u_{(\mathbf{0})}(t,x):=\int_{\mathbb{R}}p(t,x-y)u_0(y)\;dy.$$


\subsection{$u^{\mathtt{FK}}$ and $u^{\mathtt{MW}}$}

We recall the multiple Wiener solution of \eqref{eq:1dSHE}:
\begin{align*}
u^{\mathtt{MW}}(t,x)&=\sum_{n=0}^{\infty}I_n\left(F_n^{\mathtt{MW}}(t,x)\right),
\end{align*}
where \begin{align*}
&F_0^{\mathtt{MW}}(t,x)=u_{(\mathbf{0})}(t,x);\\
&F_n^{\mathtt{MW}}(t,x;y_1,\dots,y_n)\\
&=\frac{1}{n!}\int_{[0,t]^n} p\left(t-r_{\rho(n)},x-y_{\rho(n)}\right)\times \cdots
\times  p\left(r_{\rho(2)}-r_{\rho(1)},y_{\rho(2)}-y_{\rho(1)}\right)u_{(\mathbf{0})}(r_{\rho(1)},x_{\rho(1)})\; d\mathbf{r},\ n\geq 1,
\end{align*}
and
$\rho$ denotes the permutation of $\{1,\dots,n\}$ such that $0<r_{\rho(1)}<\cdots<r_{\rho(n)}<t$.

Moreover, the Feynman-Kac solution is given by
\begin{align*}
u^{\mathtt{FK}}(t,x)&=\mathbb E^B\left[u_0(B_t^x)\exp\{\Psi_{t,x}\}\right],
\end{align*}
where
$\{B_t^x\}_{t\geq 0}:=\{B_t+x\}_{t\geq 0}$, for fixed $(t,x)\in [0,T]\times\mathbb R$, and 
\begin{align}\label{eq:Psi}
\Psi_{t,x}(\omega;\widetilde{\omega})=\int_{\mathbb R} L_y^x(t;\widetilde{\omega})dW(y;\omega)-\frac{1}{2}\int_{\mathbb R}|L_y^x(t;\widetilde{\omega})|^2 dy,\quad \mbox{$\mathbb{P}^W\otimes \mathbb{P}^B$-almost surely},
\end{align}
where $L_a^x(t)$ denotes the local time of $\{B_s^x\}_{s\geq 0}$ at level $a$ and time $t$. Note that the stochastic integral in \eqref{eq:Psi} is well-defined since the function $y\mapsto L_y^x(t;\widetilde{\omega})$ is square integrable for each fixed $(t,x,\widetilde{\omega})$. From now on, we will omit the explicit dependence on $(\omega,\widetilde{\omega})$ unless there is a risk of confusion.
Furthermore, we notice that by definition $\exp\{\Psi_{t,x}\}=\mathcal E(L^x(t))$, i.e. the stochastic exponential (e.g. \cite{HOUZ2010}) of the Brownian local time.

Using Lemma \ref{lemma: chaos decomposition kernel} and the fact that $D\mathcal E(L^x(t))=\mathcal E(L^x(t))L^x(t)$ and $\mathbb E^W\left[\mathcal E(L^x(t))\right]=1$, we can rewrite $u^{\mathtt{FK}}$ as
\begin{align*}
u^{\mathtt{FK}}(t,x)=\sum_{n=0}^{\infty}I_n\left(F_n^{\mathtt{FK}}(t,x)\right),
\end{align*}
where 
\begin{align*}
F_0^{\mathtt{FK}}(t,x)&=u_{(\mathbf{0})}(t,x);\\
F_n^{\mathtt{FK}}(t,x;y_1,\dots,y_n)&=\frac{1}{n!}\mathbb E^B\left[(L^x(t))^{\otimes n} (y_1,\dots,y_n)u_0(B_t^x)\right],\ n\geq 1.
\end{align*}
We will directly prove $u^{\mathtt{FK}}=u^{\mathtt{MW}}$ by showing $F_n^{\mathtt{FK}}=F_n^{\mathtt{MW}}$ for all $n\geq 0$.

For each $t\in [0,T]$ and $x\in \mathbb{R}$, it is clear that $F_0^{\mathtt{FK}}(t,x)=F_0^{\mathtt{MW}}(t,x)$.
We now let $n\geq 1$. Then, by Fubini lemma,
\begin{align*}
F_n^{\mathtt{FK}}(t,x;y_1,\dots,y_n)=\frac{1}{n!}\int_{[0,t]^n}\mathbb{E}^B\left[\delta_0(B_{s_1}^x-y_1)\cdots\delta_0(B_{s_n}^x-y_n)u_0(B_t^x)\right]\; d\mathbf{s},
\end{align*}
where $\delta_x(y):=\displaystyle\lim_{\varepsilon\to 0}(\pi\varepsilon)^{-1/2}e^{-|x-y|^2/\varepsilon}$, $x,y\in \mathbb{R}$ as the Dirac-delta function in the sense of distribution.
Let $\sigma$ be the permutation of $\{1,\dots,n\}$ such that $0<s_{\sigma(1)}<\cdots<s_{\sigma(n)}<t$. Then, we have
\begin{align*}
F_n^{\mathtt{FK}}&(t,x;y_1,\dots,y_n)=\frac{1}{n!}\int_{[0,t]^n}\mathbb{E}^B\left[\delta_0(B_{s_{\sigma(1)}}^x-y_{\sigma(1)})\cdots\delta_0(B_{s_{\sigma(n)}}^x-y_{\sigma(n)})u_0(B_t^x)\right]\; d\mathbf{s},\\
&=\frac{1}{n!}\int_{[0,t]^n}\mathbb{E}^B\left[\delta_0(B_{s_{\sigma(1)}}^x-y_{\sigma(1)})\cdots\delta_0(B_{s_{\sigma(n)}}^x-y_{\sigma(n)})\mathbb{E}^B\left[u_0(B_t^x)|\mathcal F_{s_{\sigma(n)}}\right]\right]\; d\mathbf{s},\\
&=\frac{1}{n!}\int_{[0,t]^n}\mathbb{E}^B\left[\delta_0(B_{s_{\sigma(1)}}^x-y_{\sigma(1)})\cdots\delta_0(B_{s_{\sigma(n)}}^x-y_{\sigma(n)})u_{(\mathbf{0})}\left(t-s_{\sigma(n)},B_{\sigma(n)}^x\right)\right]\; d\mathbf{s}.
\end{align*}
We know that for any $t\geq s$ and $f\in L^{\infty}(\mathbb{R})$, 
\begin{align*}
\mathbb E\left[\delta_0(B_t^x-y)f(B_t^x)|\mathcal F_s\right]&=\int_{\mathbb R}p\left(t-s,B_s^x-z\right)f(z)\delta_0(z-y)dz
= p\left(t-s,B_s^x-y\right)f(y).
\end{align*}
If we use the identity iteratively, we can obtain 
\begin{align}\label{eq:Fn-FK}
F_n^{\mathtt{FK}}(t,x;y_1,\dots,y_n)&=\frac{1}{n!}\int_{[0,t]^n} p\left(s_{\sigma(1)},y_{\sigma(1)}-x\right)\times \cdots\nonumber\\
& \times  p\left(s_{\sigma(n)}-s_{\sigma(n-1)},y_{\sigma(n)}-y_{\sigma(n-1)}\right)u_{(\mathbf{0})}(t-s_{\sigma(n)},x_{\sigma(n)})\; d\mathbf{s}.
\end{align}
Let $r_i=t-s_i$ for $i=1,\dots, n$ and $\rho(1)=\sigma(n)$, $\rho(2)=\sigma(n-1)$, $\dots$, $\rho(n)=\sigma(1)$. Then, it is clear that 
$0<r_{\rho(1)}<\cdots<r_{\rho(n)}<t$,
and we can rewrite \eqref{eq:Fn-FK} as
\begin{align*}
F_n^{\mathtt{FK}}(t,x;y_1,\dots,y_n)&=\frac{1}{n!}\int_{[0,t]^n} p\left(t-r_{\rho(n)},x-y_{\rho(n)}\right)\times \cdots\\
&\times  p\left(r_{\rho(2)}-r_{\rho(1)},y_{\rho(2)}-y_{\rho(1)}\right)u_{(\mathbf{0})}(r_{\rho(1)},x_{\rho(1)})\; d\mathbf{r}=F_n^{\mathtt{MW}}(t,x;y_1,\dots,y_n),
\end{align*}
and thus $u^{\mathtt{FK}}(t,x)=u^{\mathtt{MW}}(t,x)$.

\subsection{$u^{\mathtt{FK}}$ and $u^{\mathtt{CS}}$}\label{sec:MW-CS}

Let us now show $u^{\mathtt{FK}}=u^{\mathtt{CS}}$ directly. Recall
\begin{align*}
u^{\mathtt{CS}}(t,x)=\sum_{\alpha\in\mathcal J} u^{\mathtt{CS}}_{\alpha}(t,x)\xi_{\alpha},
\end{align*}
where $u^{\mathtt{CS}}_{(\mathbf{0})}(t,x)=u_{(\mathbf{0})}(t,x)$, and for $|\alpha|=n\geq 1$,
\begin{align*}
&u^{\mathtt{CS}}_{\alpha}(t,x)=\sqrt{n!}\int_{\mathbb R^n} F_n^{\mathtt{CS}}(t,x;y_1,\dots,y_n)\mathfrak{e}_{\alpha}(y_1,\dots,y_n)\; d\mathbf{y},\quad \alpha\in \mathcal J_n=\{\alpha\in\mathcal J:|\alpha|=n\},\\
&F_n^{\mathtt{CS}}(t,x;y_1,\dots,y_n)=\int_{\mathbb T^n_{[0,t]}}p(t-s_n,x-y_n)\cdots p(s_2-s_1,y_2-y_1)u_{(\mathbf{0})}(s_1,y_1)\; d\mathbf{s},
\end{align*}
$\mathbb T^n_{[0,t]}=\{0\leq s_1\leq\cdots\leq s_n\leq t\}$,
and $\{\mathfrak{e}_{\alpha},\alpha\in \mathcal{J}_n\}$ is an orthonormal basis of $L^2_{sym}(\mathbb{R}^n):=$ the symmetric part of $L^2(\mathbb{R}^n)$.
Specifically, 
\begin{align}\label{eq:ealpha}
\mathfrak{e}_{\alpha}=\frac{1}{\sqrt{n!\alpha!}}\sum_{\sigma\in \mathcal{P}_n} e_{k_{\sigma(1)}}(y_1)\cdots e_{k_{\sigma(n)}}(y_n),
\end{align}
where $k_{\alpha}=(k_1,\dots,k_n)$ be its characteristic vector for any $\alpha\in \mathcal J_n$ (e.g. \cite[Section 2]{KL2017}). 

We have $u^{\mathtt{CS}}_{(\mathbf{0})}(t,x)=u^{\mathtt{FK}}_{(\mathbf{0})}(t,x)$ for $(t,x)\in [0,T]\times \mathbb{R}$.
For $n\geq 1$, we have
\begin{align*}
\sum_{\alpha\in\mathcal J_n}u^{\mathtt{CS}}_{\alpha}(t,x)\mathfrak{e}_{\alpha}=\sqrt{n!}\sum_{\alpha\in\mathcal J_n} \left\langle F^{\mathtt{CS}}_n(t,x),\mathfrak{e}_{\alpha}\right\rangle_{L^2(\mathbb R^n)}\mathfrak{e}_{\alpha}.
\end{align*}
In fact, it is equal to the orthogonal projection of $F^{\mathtt{CS}}_n$ on $L^2_{sym}(\mathbb{R}^n)$, and thus
\begin{align*}
\sum_{\alpha\in\mathcal J_n}u^{\mathtt{CS}}_{\alpha}(t,x)\frac{\mathfrak{e}_{\alpha}}{\sqrt{n!}}=\operatorname{Sym}\left(F^{\mathtt{CS}}_n(t,x)\right)=:\widetilde{F^{\mathtt{CS}}_n(t,x)}.
\end{align*}
Note that 
\begin{align*}
\widetilde{F^{\mathtt{CS}}_n(t,x)}=\frac{1}{n!}\sum_{\sigma\in \mathcal P_{n}}\int_{\mathbb T^n_{[0,t]}}p(t-s_n,x-y_{\sigma(n)}) \cdots p(s_2-s_1,y_{\sigma(2)}-y_{\sigma(1)})u_{(\mathbf{0})}(s_1,y_{\sigma(1)})\; d\mathbf{s}.
\end{align*}

Now we take the $n$-fold Wiener integral $I_n(\bullet)$ on both sides to get
\begin{align}\label{eq:I_n}
\sum_{\alpha\in\mathcal J_n}u^{\mathtt{CS}}_{\alpha}(t,x)I_n\left(\frac{\mathfrak{e}_{\alpha}}{\sqrt{n!}}\right)=I_n\left(\widetilde{F^{\mathtt{CS}}_n(t,x)}\right).
\end{align}
There's a result due to It\^o (see for instance \cite[equation 2.2.29]{HOUZ2010}) stating that
\begin{align}\label{eq:Inxia}
I_n\left(\displaystyle\operatorname{Sym}\bigotimes_{j=1}^{\infty} e_{j}^{\otimes \alpha_j}\right)=\prod_{j=1}^{\infty}H_{\alpha_j}\left(I_1(e_{j})\right) \quad \mbox{and thus}\quad I_n\left(\frac{\mathfrak{e}_{\alpha}}{\sqrt{n!}}\right)=\xi_{\alpha}.
\end{align}
Plugging this into \eqref{eq:I_n}, we obtain
\begin{align*}
\sum_{\alpha\in\mathcal J_n}u^{\mathtt{CS}}_{\alpha}(t,x)\xi_{\alpha}=I_n\left(\widetilde{F^{\mathtt{CS}}_n(t,x)}\right),
\end{align*}
and thus 
\begin{align*}
u^{\mathtt{CS}}(t,x)=u^{\mathtt{CS}}_{(\mathbf{0})}(t,x)+\sum_{n=1}^{\infty}\sum_{\alpha\in\mathcal J_n}u^{\mathtt{CS}}_{\alpha}(t,x)\xi_{\alpha}= u^{\mathtt{CS}}_{(\mathbf{0})}(t,x)+ \sum_{n=1}^{\infty}I_n\left(\widetilde{F^{\mathtt{CS}}_n(t,x)}\right).
\end{align*}
To have that $u^{\mathtt{CS}}=u^{\mathtt{FK}}$, it only remains to show that $\widetilde{F^{\mathtt{CS}}_n(t,x)}=F^{\mathtt{FK}}_n(t,x)$ for $n\geq 1$. Indeed,
\begin{align*}
F^{\mathtt{FK}}_n(t,x)&=\frac{1}{n!}\int_{[0,t]^n}\mathbb{E}^B\left[\delta(B_{s_1}^x-y_1)\cdots\delta(B_{s_n}^x-y_n)u_0(B_t^x)\right]\; d\mathbf{s}\\
&=\frac{1}{n!}\frac{1}{n!}\sum_{\sigma\in \mathcal P_{n}}\int_{[0,t]^n}\mathbb{E}^B\left[\delta\left(B_{s_1}^x-y_{\sigma(1)}\right)\cdots\delta\left(B_{s_n}^x-y_{\sigma(n)}\right)u_0(B_t^x)\right]\; d\mathbf{s}\\
&=\frac{1}{n!}\sum_{\sigma\in \mathcal P_{n}}\int_{\mathbb T^n_{[0,t]}}\mathbb{E}^B\left[\delta_{y_{\sigma(1)}}\widehat{\otimes}\cdots\widehat{\otimes}\delta_{y_{\sigma(n)}}(B_{s_1}^x,\dots, B_{s_n}^x)u_0(B_t^x)\right]\; d\mathbf{s}\\
&=\frac{1}{n!}\sum_{\sigma\in \mathcal P_{n}}\int_{\mathbb T^n_{[0,t]}}\int_{\mathbb R^n}\left[\delta_{y_{\sigma(1)}}\widehat{\otimes}\cdots\widehat{\otimes}\delta_{y_{\sigma(n)}}\right](y_1,\dots,y_n)p(s_n-s_{n-1},y_n-y_{n-1})\times\cdots \\
&\quad\quad\quad\quad\quad\quad\quad\quad\quad \times p(s_2-s_1,y_2-y_1)p(s_1,y_1-x)u_{(\mathbf{0})}(t-s_n,y_n)\; d\mathbf{y}\; d\mathbf{s}\\
&=\frac{1}{n!}\sum_{\sigma\in \mathcal P_{n}}\int_{\mathbb T^n_{[0,t]}}p(s_n-s_{n-1},y_{\sigma(n)}-y_{\sigma(n-1)})\cdots p(s_1,y_{\sigma(1)}-x)u_{(\mathbf{0})}(t-s_n,y_{\sigma(n)})\; d\mathbf{s}\\	
&=\widetilde{F^{\mathtt{CS}}_n(t,x)}\quad \mbox{after rearranging the variables.}
\end{align*}


\section{Basic regularity of $u$}\label{sec:FKsol}

We again assume that $u_0\in L^{\infty}(\mathbb{R})$ so that $u=u^{\mathtt{FK}}=u^{\mathtt{MW}}=u^{\mathtt{CS}}$ and we will denote $\|\bullet\|_{\infty}:=\|\bullet\|_{L^{\infty}(\mathbb R)}$ In this section,
we will provide a few basic regularity of $u$ using the Feynman-Kac representation.

\begin{theorem}
For every $(t,x)\in [0,T]\times\mathbb R$, 
	$$u(t,x)\in \mathcal G.$$
\end{theorem}
\begin{proof}
	We have for $\phi\in \mathscr{S}(\mathbb{R})$,
	\begin{align}\label{eq:Stu-v1}
	S(u(t,x))(\phi)&=\mathbb E^W\left[u(t,x)\mathcal E(\phi)\right]
	=\mathbb E^W\mathbb E^B\left[u_0(B_t^x)\mathcal E(L^x(t))\mathcal E(\phi)\right]\nonumber\\
	&=\mathbb E^{B}\left[u_0(B_t^x)\exp\left(\int_{\mathbb R} L^x_y(t)\phi(y)dy\right)\right],
	\end{align}
	where the last equality comes from Fubini Lemma and \cite[Theorem 5.13]{K2018}.
	
	Let $P_m:\mathscr{S}'(\mathbb R)\to\mathscr{S}(\mathbb R)$ be the orthogonal projection of $\mathscr{S}'(\mathbb R)$ on $\operatorname{span}\{e_1,\dots,e_m\}$, $m\geq 1$.
	Then for $\eta\in \mathscr S_{\mathbb C}'(\mathbb R):= \mathscr S'(\mathbb R)\oplus i\mathscr S'(\mathbb R)$ and $\lambda\in\mathbb R$, we have
	\begin{align*}
	S(u(t,x))(\lambda P_m\eta)=\mathbb{E}^{B}\left[u_0(B_t^x)\exp\bigg\{\int_{\mathbb{R}}\lambda P_m\eta(y)L_y^x(t)dy\bigg\}\right].
	\end{align*}
	Using the fact that for $\phi,\eta\in \mathscr S'_{\mathbb C}(\mathbb R)$ it holds  that $\langle \eta,P_m\phi \rangle=\langle \phi,P_m\eta\rangle$, we can write
	\begin{align*}
	S(u(t,x))(\lambda P_m\eta)=\mathbb{E}^{B}\left[u_0(B_t^x)\exp\bigg\{\lambda\int_{\mathbb{R}} \eta(y)(P_mL^x(t))(y)dy\bigg\}\right]
	\end{align*}
	and
	\begin{align*}
	|S(u(t,x))(\lambda P_m\eta)|^2=\left|\mathbb{E}^{B}\left[u_0(B_t^x)\exp\bigg\{\lambda\int_{\mathbb{R}} (\eta_1(y)+i\eta_2(y))(P_mL^x(t))(y)dy\bigg\}\right]\right|^2.
	\end{align*}
	By Jensen's inequality,	we have
	\begin{align*}
	|S(u(t,x))(\lambda P_m\eta)|^2\leq \|u_0\|_{\infty}^2\mathbb{E}^{B}\left[\left|\exp\bigg\{\lambda\int_{\mathbb{R}} (\eta_1(y)+i\eta_2(y))(P_mL^x(t))(y)dy\bigg\}\right|^2\right].
	\end{align*}
	Since $|z_1\cdot z_2|=|z_1|\cdot|z_2|$ for any $z_1,z_2\in \mathbb{C}$,
	\begin{align*}
	&|S(u(t,x))(\lambda P_m\eta)|^2\\ &\leq\|u_0\|_{\infty}^2\mathbb{E}^{B}\left[\left|\exp\bigg\{\lambda\int_{\mathbb{R}} \eta_1(y)(P_mL^x(t))(y)dy\bigg\}\right|^2\left|\exp\bigg\{i\lambda\int_{\mathbb{R}}\eta_2(y)(P_mL^x(t))(y)dy\bigg\}\right|^2\right]\nonumber\\
	&= \|u_0\|_{\infty}^2\mathbb{E}^B\left[\exp\bigg\{2\lambda\langle \eta_1,P_mL^x(t)\rangle_{L^2(\mathbb R)}\bigg\}\right].
	\end{align*}
	Thus,
	\begin{align*}
	\int_{\mathscr S_{\mathbb C}'(\mathbb R)}|S(u(t,x))(\lambda P_m\eta)|^2\nu(d\eta)\leq \|u_0\|_{\infty}^2\mathbb E^{B}\left[\int_{\mathscr S_{\mathbb C}'(\mathbb R)}\exp\bigg\{2\lambda\langle \eta_1,P_mL^x_{\bullet}(t)\rangle_{L^2(\mathbb R)}\bigg\}\nu(d\eta)\right],
	\end{align*}
	where the measure $\nu$ is given by the product measure $\mu_{\frac{1}{2}}\otimes\mu_{\frac{1}{2}}$, where $\mu_{\frac{1}{2}}$ is the measure on $(\Omega,\mathcal B)$ with the characteristic function given by:
	\begin{align*}
	\int_{\mathscr{S}'(\mathbb R)}e^{i\langle\omega,\varphi\rangle}d\mu(\omega)=e^{-\frac{1}{4}|\varphi|_0^2},\quad\quad\varphi\in\mathscr{S}(\mathbb R).
	\end{align*}
	It is clear that for any $\varphi \in\mathscr{S}(\mathbb R)$, $\mu_{\frac{1}{2}}\circ\left\langle\bullet,\frac{\varphi}{|\varphi|_0}\right\rangle^{-1}$ is a centered Gaussian measure with variance $1/2$ as in \cite[Lemma 2.1.2]{HOUZ2010}. Therefore,
	\begin{align*}
	\int_{\mathscr S_{\mathbb C}'(\mathbb R)}|S(u(t,x))(\lambda P_m\eta)|^2\nu(d\eta)&\leq \|u_0\|_{\infty}^2 \mathbb E^B\left[\frac{1}{\sqrt{\pi}}\int_{\mathbb R}e^{2\lambda y |P_mL^x(t)|_0}e^{-y^2}dy\right]\\
	&=\|u_0\|_{\infty}^2\mathbb E^B\left[e^{\lambda^2|P_mL^x(t)|^2_0}\right].
	\end{align*}
	Finally, we obtain (by \cite[page 178]{GJF1994})
	\begin{align*}
	\lim_{m\to\infty}\int_{\mathscr S_{\mathbb C}'(\mathbb R)}|S(u(t,x))(\lambda P_m\eta)|^2\nu(d\eta)\leq \|u_0\|_{\infty}^2 \mathbb E\left[e^{\lambda^2\int_{\mathbb R}|L_y^x(t)|^2dy}\right]<\infty,\; \forall \lambda\in\mathbb R,
	\end{align*}
	which implies by \cite[Corollary 5.1]{GMN2021}, $u(t,x)$ belongs to $\mathcal G$. 
\end{proof}

Next, we state the basic H\"older regularity of $u$ both in time and space.

\begin{theorem}
Let $0<\varepsilon<1/2$ be arbitrary and $C$ be a constant.
	\begin{itemize}
		\item[(i)] Assume that $u_0\equiv C$. Then, 
		$$u\in \mathcal{C}^{3/4-\varepsilon,1/2-\varepsilon}([0,T]\times \mathbb{R}).$$
		\item[(ii)] Assume that $u_0\not\equiv C$ and $u_0\in L^{\infty}(\mathbb{R})$ is (globally) Lipschitz continuous on $\mathbb{R}$.
			Then, 
		$$u\in \mathcal{C}^{1/2-\varepsilon,1/2-\varepsilon}([0,T]\times \mathbb{R}).$$
	\end{itemize}
\end{theorem}
\begin{proof}
	Let $\vvvert\bullet\vvvert_p:=\left(\mathbb{E}^W\mathbb{E}^B|\bullet|^p\right)^{1/p}$ be the norm on the Banach space $L^p(\mathbb{P}^W\otimes \mathbb{P}^B)$ for $p\geq 1$. From \cite{S2021}, we have 
	\begin{align}\label{eq:expPsibdd}
	\displaystyle\sup_{(t,x)\in [0,T]\times \mathbb{R}} \vvvert \mbox{exp}\left\{\Psi_{t,x}\right\}\vvvert_p<\infty.
	\end{align} 
	Also, \eqref{eq:Psi} conditional on $B$, becomes
	\begin{align}\label{eq:Psi-Gaussian}
	\Psi_{t,x}\sim N\left(-\frac{1}{2}\int_{\mathbb R}|L_a(t)|^2da,\int_{\mathbb R}|L_a(t)|^2da\right),
	\end{align} 
	where $N\left(\mu,\sigma^2\right)$ denotes the Gaussian random variable with mean $\mu$ and variance $\sigma^2$.
	
	(i) Let $u_0$ be a constant. Then,
		using the fact $|e^x-e^y|\leq (e^x+e^y)|x-y|$ for all $x,y\in \mathbb{R}$, Cauchy-Schwarz inequality, Minkowski inequality, \eqref{eq:expPsibdd}, and \eqref{eq:Psi-Gaussian}, we can obtain for $p\geq 2$.
	\begin{align*}
	\mathbb{E}^W\left[|u(t,x)-u(s,y)|^p\right]&\leq c_p\bigg\{\mathbb E^B\mathbb E^W\left[|\Psi_{t,x}-\Psi_{s,y}|^2\right]^{1/2}\bigg\}^p
	=c_p\vvvert \Psi_{t,x}-\Psi_{s,y}\vvvert_2^p,
	\end{align*}
	for some $c_p>0$. Also, the triangular inequality implies
	\begin{align*}
	\mathbb E^W\left[|u(t,x)-u(s,y)|^p\right]\leq c_p\bigg\{\vvvert \Psi_{t,x}-\Psi_{t,y}\vvvert_2+\vvvert \Psi_{t,y}-\Psi_{s,y}\vvvert_2\bigg\}^p=:c_p\left(A_1^{1/2}+A_2^{1/2}\right)^p.
	\end{align*}
	Let us now work with 
	\begin{align*}
	A_1=\vvvert \Psi_{t,x}-\Psi_{t,y}\vvvert_2^2=\mathbb E^B\mathbb E^W\left[\Psi_{t,x}^2-2\Psi_{t,x}\Psi_{t,y}+\Psi_{t,y}^2\right].
	\end{align*}
	By \eqref{eq:Psi-Gaussian}, we have
	\begin{align*}
	A_1=\mathbb E^B\left[2\int_{\mathbb R}|L_a(t)|^2da+\frac{1}{2}\left(\int_{\mathbb R}|L_a(t)|^2da\right)^2-2\mathbb E^W[\Psi_{t,x}\Psi_{t,y}]\right].
	\end{align*}
	Recall the Dirac-delta function $\delta_x(y)=\displaystyle\lim_{\varepsilon\to 0}(\pi\varepsilon)^{-1/2}e^{-|x-y|^2/\varepsilon}$, $x,y\in \mathbb{R}$.
	Since
	\begin{align*}
	\mathbb{E}^B\mathbb E^W\left[\Psi_{t,x}\Psi_{t,y}\right]=\mathbb{E}^B\left[\int_0^t\int_0^t\delta_0(B_u-B_r-(x-y))dudr+\frac{1}{4}\left(\int_{\mathbb{R}}|L_a(t)|^2da\right)^2\right],
	\end{align*}
	we get
	\begin{align*}
	A_1=\mathbb E^B\left[2\int_{\mathbb R}|L_a(t)|^2da-2\int_0^t\int_0^t\delta_0(B_u-B_r-(x-y))dudr\right].
	\end{align*}
	The next step is done rigorously (See \cite{H2016} for instance) by the translation invariant property of the Lebesgue measure:
	\begin{align*}
	A_1=\mathbb E^B\left[\int_{\mathbb R}|L_{a-x}(t)|^2da-2\int_0^t\int_0^t\delta_0(B_u-B_r-(x-y))dudr+\int_{\mathbb R}|L_{a-y}(t)|^2da\right],
	\end{align*}
	and this yields, by \cite[Proposition 9.2]{H2016},
	\begin{align*}
	A_1=\mathbb{E}^B\left[\int_{\mathbb R} \left|L_{a-x}(t)-L_{a-y}(t)\right|^2da\right]=4t|x-y|+O(|x-y|^2),
	\end{align*}
	which implies $u(t,\bullet)$ is almost H\"older $1/2$ continuous uniformly for all $t\in [0,T]$. 
	
	On the other hand, let us compute, for $0\leq s\leq t\leq T$,
	\begin{align*}
	A_2&=\mathbb E^B\mathbb E^W\left[\Psi_{t,y}^2-2\Psi_{t,y}\Psi_{s,y}+\Psi_{s,y}^2\right]\\
	&=\mathbb{E}^B\bigg[\int_{\mathbb{R}}|L_{a}(t)|^2da+\frac{1}{4}\left(\int_{\mathbb{R}}|L_{a}(t)|^2da\right)^2+\int_{\mathbb{R}}|L_{a}(s)|^2da+\frac{1}{4}\left(\int_{\mathbb{R}}|L_{a}(s)|^2da\right)^2\\
	&\quad-2\mathbb{E}^W\left[\Psi_{t,y}\Psi_{s,y}\right]\bigg].
	\end{align*}
	Since
	\begin{align*}
	\mathbb{E}^B\mathbb{E}^W\left[\Psi_{t,y}\Psi_{s,y}\right]=\mathbb{E}^B\left[\int_0^t\int_0^s\delta_0(B_u-B_r)dudr+\frac{1}{4}\left(\int_{\mathbb{R}}|L_{a}(t)|^2da\right)\left(\int_{\mathbb{R}}|L_{a}(s)|^2da\right)\right],
	\end{align*}
	we have
	\begin{align*}
	A_2&=\mathbb{E}^B\left[\int_{\mathbb{R}}|L_{a}(t)|^2da-2\int_0^t\int_0^s\delta_0(B_r-B_z)drdz+\int_{\mathbb{R}}|L_{a}(s)|^2da\right]\\
	&\quad +\frac{1}{4}\mathbb{E}^B\left[\left(\int_{\mathbb{R}}|L_{a}(t)|^2da\right)^2-2\left(\int_{\mathbb{R}}|L_{a}(t)|^2da\right)\left(\int_{\mathbb{R}}|L_{a}(s)|^2da\right)+\left(\int_{\mathbb{R}}|L_{a}(s)|^2da\right)^2\right]\\
	&=\mathbb{E}^B\left[\int_{\mathbb{R}}|L_{a}(t)|^2da-2\int_0^t\int_0^s\delta_0(B_r-B_z)drdz+\int_{\mathbb{R}}|L_{a}(s)|^2da\right]\\
	&\quad +\frac{1}{4}\mathbb{E}^B\left(\int_{\mathbb{R}}\left(|L_a(t)|^2-|L_a(s)|^2\right)da\right)^2.
	\end{align*}
	We note that $$\mathbb{E}^B\left[\int_{\mathbb{R}}|L_a(t)|^2da\right]=\mathbb{E}^B\left[\int_0^t\int_0^t \delta_0(B_r-B_z)drdz\right],$$
	which implies
	\begin{align*}
	A_2&=\mathbb{E}^B\left[\int_0^t\int_0^t \delta_0(B_r-B_z)drdz-2\int_0^t\int_0^s \delta_0(B_r-B_z)drdz+\int_0^s\int_0^s \delta_0(B_r-B_z)drdz\right]\\
	&\quad +\frac{1}{4}\mathbb{E}^B\left(\int_0^t\int_0^t \delta_0(B_r-B_z)drdz-\int_0^s\int_0^s \delta_0(B_r-B_z)drdz\right)^2\\
	&=\mathbb{E}^B\left[\int_s^t\int_s^t \delta_0(B_r-B_z)drdz\right]
	+\frac{1}{4}\mathbb{E}^B\left(\int_0^t\int_0^t \delta_0(B_r-B_z)drdz-\int_0^s\int_0^s \delta_0(B_r-B_z)drdz\right)^2\\
	&=:A_3+A_4.
	\end{align*}
	We can easily compute $A_3$:
	\begin{align*}
	A_3&=\int_s^t\int_s^t (2\pi|r-z|)^{-1/2}drdz=C(t-s)^{3/2}\quad \mbox{for some $C>0$ independent of $x$}.
	\end{align*}
	For $A_4$, we have
	\begin{align*}
	4A_4&=\int_s^t\int_s^t\int_s^t\int_s^t \mathbb{E}^B\left(\delta_0(B_z-B_r) \delta_0(B_q-B_p)\right)dpdqdrdz\\
	&\stackrel{\text{symmetry}}{=}4!\int_s^t\int_s^z\int_s^r\int_s^q \mathbb{E}^B\left[\delta_0(B_z-B_r)\right]\mathbb E^B\left[ \delta_0(B_q-B_p)\right]dpdqdrdz\\
	&=4!\int_s^t\int_s^z\int_s^r\int_s^q \frac{1}{\sqrt{2\pi(z-r)}} \frac{1}{\sqrt{2\pi(q-p)}}dpdqdrdz\\
	&\leq 4!\int_s^t\int_s^z\frac{1}{\sqrt{2\pi(z-r)}}\left(\int_s^{t}\int_s^q  \frac{1}{\sqrt{2\pi(q-p)}}\right)dpdqdrdz\\
	&=C (t-s)^{3}\quad \mbox{for some $C>0$ independent of $x$}.
	\end{align*}
	Combining all together, we obtain
	$$A_2\leq C|t-s|^{3/2},$$
	which implies $u(\bullet,x)$ is almost H\"older $3/4$ continuous uniformly for all $x\in \mathbb{R}$. 

(ii) If $u_0$ is not a constant function on $\mathbb{R}$, then we have
	\begin{align*}
		\mathbb E^W\left[|u(t,x)-u(s,y)|^p\right]&=\mathbb{E}^W\left|\mathbb{E}^B\left(u_0(x+B_t)\mbox{exp}\left(\Psi_{t,x}\right)\right)-\mathbb{E}^B\left(u_0\left(y+B_s\right)\mbox{exp}(\Psi_{s,y})\right)\right|^p\\
		&=\mathbb{E}^W\Big|\mathbb{E}^B\left(\left(u_0(x+B_t)-u_0(y+B_s)\right)\mbox{exp}(\Psi_{t,x})\right)\\
		&\qquad\qquad +\mathbb{E}^B\left(u_0(y+B_s)\left(\mbox{exp}(\Psi_{t,x})-\mbox{exp}(\Psi_{s,y})\right)\right)\Big|^p.
	\end{align*}
Since $|f+g|^p\leq 2^{p-1}\left(|f|^p+|g|^p\right)$ for $p\geq 1$,
	\begin{align*}
	\mathbb E^W\left[|u(t,x)-u(s,y)|^p\right]&\leq 2^{p-1} \bigg(\mathbb{E}^W\left|\mathbb{E}^B\left(\left(u_0\left(x+B_t\right)-u_0\left(y+B_s\right)\right)\mbox{exp}(\Psi_{t,x})\right)\right|^p\\
	&\qquad\qquad +\mathbb{E}^W\left|\mathbb{E}^B\left(u_0\left(y+B_s\right)\left(\mbox{exp}(\Psi_{t,x})-\mbox{exp}(\Psi_{s,y})\right)\right)\right|^p
	\bigg)\\
	&=:2^{p-1}\left(\bar{A}_1+\bar{A}_2\right).
	\end{align*}
	For $\bar{A}_1$, by Cauchy-Schwarz inequality,
	\begin{align*}
	\bar{A}_1&\leq  \left(\mathbb{E}^B\left(u_0(x+B_t)-u_0(y+B_s)\right)^2\right)^{p/2}\cdot\mathbb{E}^W\left(\mathbb{E}^B\left(\mbox{exp}\left(2\Psi_{t,x}\right)\right)\right)^{p/2}.
	\end{align*}
	Since $u_0$ is Lipschitz continuous on $\mathbb{R}$, we have 
	\begin{align*}
	\bar{A}_1&\leq 
	\left(|t-s|+(x-y)^2\right)^{p/2}\cdot \mathbb{E}^W\left(\mathbb{E}^B\left(\mbox{exp}\left(2\Psi_{t,x}\right)\right)\right)^{p/2}\\
	&\leq
	\left(|t-s|^{1/2}+|x-y|\right)^{p}\cdot\mathbb{E}^W\left(\mathbb{E}^B\left(\mbox{exp}\left(2\Psi_{t,x}\right)\right)\right)^{p/2}.
	\end{align*}
	By Minkowski inequality, H\"older inequality for $p\geq 2$, and \eqref{eq:expPsibdd}, we also have
	\begin{align*}
	\bar{A}_1\leq \left(|t-s|^{1/2}+|x-y|\right)^{p}\cdot\mathbb{E}^W\mathbb{E}^B\left(\mbox{exp}\left(p\Psi_{t,x}\right)\right)<\infty.
	\end{align*}
	For $\bar{A}_2$, since $u_0\in L^{\infty}(\mathbb{R})$, we can apply the same argument in (i). As a result, we can say that $u$ is H\"older continuous almost $1/2$ both in time and space.
\end{proof}

As we argued in the introduction, one expects that we can still improve the spatial regularity of $u$. We will derive our desired result in Section \ref{sec:uFKx}.


\section{The spatial derivative of $u$}\label{sec:uFKx}

As we anticipated in the introduction, we expect that $u(t,\bullet)\in C^{3/2-\varepsilon}(\mathbb{R})$ for any small $\varepsilon>0$. To verify this assertion, we first compute the spatial derivative of $u$ using the Feynman-Kac representation and then find its chaos expansion to see if it is well-defined in $\mathcal{G}$ and to get the optimal spatial regularity of $u$.

Let us start with a useful Lemma. The following result will serve as a key idea for finding $\partial_x u(t,x)$.  
	\begin{lemma}\label{lemma:Phi}
	For fixed $(t,x)$ the map $\widetilde{\omega}\ni\widetilde{\Omega}\mapsto \widetilde{\Phi}_{t,x}(\widetilde{\omega})\in (S)^*$ given by
	\begin{align*}
	\widetilde{\Phi}_{t,x}(\widetilde{\omega})=	\mathcal E(L^x(t;\widetilde{\omega}))\diamond \left[u_0'(B_t^x(\widetilde{\omega}))+u_0(B_t^x)I_1\left(\partial_xL^x(t;\widetilde{\omega})\right)\right],
	\end{align*}
	is Bochner integrable in $(S)^*$. Here, $\partial_xL^x(t)\in \mathscr{S}'(\mathbb R)$ denotes the pathwise distributional derivative of the local time of $\{B_t^x\}_{t\geq 0}$.
\end{lemma}
\begin{proof}
	This immediately follows from \cite[Theorem 4.51]{HKSP1990} and the facts
	\begin{enumerate}
		\item $S\left(\widetilde{\Phi}_{t,x}(\bullet)\right)(\phi)=\exp\left(\int_{0}^t\phi(B_s^x(\bullet))ds\right)\times \left[u_0'(B_t^x(\bullet))+u_0(B_t^x)\int_0^t\phi'(B_s^x(\bullet))ds\right]$ is measurable for any $\phi \in \mathscr{S}(\mathbb R)$. 
		\item 
		$\left|S\left(\widetilde{\Phi}_{t,x}(\widetilde{\omega})\right)(\phi)\right|\leq \left(\|u_0'\|_{\infty}+t\|u_0\|_{\infty}\|\phi'\|_{\infty}\right)e^{t\|\phi\|_{\infty}}\leq e^{\|u_0'\|_{\infty}+t(\|u_0\|_{\infty}\|\phi'\|_{\infty}+\|\phi\|_{\infty})}\leq K_1e^{K_2|\phi|_p^2} $
		for some $p, K_1,K_2\geq 0$. The last inequality comes from
		the equivalence between standard seminorms in the Schwartz space and the system of $p$-\textit{seminorms}.
	\end{enumerate}
\end{proof}

Recall the Feynman-Kac representation of $u$:
\begin{align*}
u(t,x)=\mathbb E^B\left[u_0(B_t^x)\mathcal E(L^x(t))\right]
\end{align*}
and recall $C^1_b(\mathbb{R})$ is the space of all bounded differentiable functions on $\mathbb{R}$ with bounded continuous derivative.
\begin{theorem}
	Assume that $u_0\in C^1_b(\mathbb R)$. Then,
	for each $t>0$, $u(t,\bullet)$ is weakly continuously differentiable in $(S)^*$, and
	\begin{align*}
	\partial_xu(t,x)&=\mathbb E^B\left[\mathcal E(L^x(t))\diamond \bigg\{u_0'(B_t^x)+u_0(B_t^x)I_1\left(\partial_xL^x(t)\right)\bigg\}\right],
	\end{align*}
	where $\partial_xL^x(t)\in \mathscr{S}'(\mathbb R)$ is the pathwise distributional derivative of $L^x(t)$, and $\mathbb{E}^B$ must be understood as a Bochner integral in $(S)^*$.
\end{theorem}
\begin{proof}
	To find $\partial_x u(t,x)$, we start by computing the S-transform of $u$. From \eqref{eq:Stu-v1}, for $\phi\in \mathscr{S}(\mathbb R)$, we have
	\begin{align}\label{eq:Stu-v2}
	S\left(u(t,x)\right)(\phi)&=\mathbb E^{B}\left[u_0(B_t^x)\exp\left(\int_{\mathbb R} L^x_y(t)\phi(y)dy\right)\right]
	=\mathbb E^B\left[u_0(B_t^x)\exp\left(\int_{0}^t\phi(B_s^x)ds\right)\right],
	\end{align}
	where the last equality follows by the occupation time formula.
	
	It is clear that $x\in \mathbb{R}\mapsto S(u(t,x))(\phi)$ is continuous for all $\phi\in \mathscr{S}(\mathbb R)$, and $\left| S\left(u(t,x)\right)(\phi)\right|\leq K_1e^{K_2|\phi|_p^2}$ for some $K_1,K_2,p>0$.
	Then, by \cite[Lemma A.1.2]{P1994}, we can see that $u(t,x)=S^{-1}(S(u(t,x)))$ is \textit{weakly }continuous in $(S)^*$ (with respect to the $x$ variable).
	
	In order to prove that $u(t,\bullet)$ is \textit{weakly } continuously differentiable in $(S)^*$ we must first take the spatial derivative on both sides of \eqref{eq:Stu-v2}. Using the fact $\phi,\phi'\in \mathscr{S}(\mathbb R)\subset L^{\infty}(\mathbb R)$, by dominated convergence theorem (DCT), we obtain
	\begin{align*}
	\partial_x S(u(t,x))(\phi)=\mathbb E^B\left[u_0'(B_t^x)\exp\left(\int_{0}^t\phi(B_s^x)ds\right)+u_0(B_t^x)\exp\left(\int_{0}^t\phi(B_s^x)ds\right)\times \int_0^t\phi'(B_s^x)ds\right],
	\end{align*}
and it is clear that the map $x\mapsto\partial_x S(u(t,x))(\phi)$ is continuous for all $\phi\in\mathscr{S}(\mathbb R)$.

We also need to show $\partial_x S\left(u(t,x)\right)$ is a U-functional (see Definition \ref{U-functional}). 
By direct computation, we can verify that, as in the proof of Lemma \ref{lemma:Phi},
\begin{align*}
\left|\partial_x S\left(u(t,x)\right)(\phi)\right|\leq K_1e^{K_2|\phi|_p^2},
\end{align*}
where $K_1, K_2, p$ are positive real constants. 
Also, it is clear that the map $z\mapsto \partial_x S(u)(z\phi+\eta)$ is entire for any $\phi,\eta\in \mathscr{S}(\mathbb R)$ and $z\in\mathbb C$. Hence, $\partial_x S\left(u(t,x)\right)$ is indeed a U-functional, and thus there exists a unique $\Phi\in (S)^*$ such that $\partial_x S(u(t,x))=S(\Phi)$; then from \cite[Lemma A.3]{P1994}, we can conclude that $u(t,\bullet)$ is \textit{weakly} continuously differentiable in the Hida distribution space $(S)^*$.
Following the aforementioned reference, the \textit{weak} spatial derivative $\partial_x u$ of $u$ is defined as the unique element in $(S)^*$ such that 
\begin{align*}
S(\partial_x u(t,x))(\phi)=\partial_{x} S(u(t,x))(\phi).
\end{align*}

Using Lemma \ref{lemma:Phi}, we can see that  $\partial_xS\left( u(t,x)\right)=\mathbb{E}^B\left(S\left(\widetilde{\Phi}_{t,x}\right)\right)$, and furthermore  we have that $\mathbb{E}^B\left(S\left(\widetilde{\Phi}_{t,x}\right)\right)=S\left(\mathbb{E}^B\left(\widetilde{\Phi}_{t,x}\right)\right)$ since the Bochner integral $\mathbb{E}^B$ and the S-transform can be interchanged (e.g. \cite[Theorem 4.51]{HKSP1990}).  
	Finally, by Theorem \ref{thm:StransformBijection}, we can conclude
	\begin{align*}
	\partial_xu(t,x)&=\mathbb E^B\left[\mathcal E(L^x(t))\diamond \bigg\{u_0'(B_t^x)+u_0(B_t^x)I_1\left(\partial_xL^x(t)\right)\bigg\}\right],
	\end{align*}
	where $\mathbb E^B$ must be understood as a Bochner integral in $(S)^*$.
\end{proof}

From this result, we can only say that $\partial_x u(t,x)\in (S)^*$ for each $(t,x)\in [0,T]\times \mathbb{R}$. But, in the following subsection, we will show that $\partial_x u(t,x)\in \mathcal{G}$ using its chaos decomposition, and furthermore, we will investigate its H\"older regularity.


\subsection{Chaos decomposition for $\partial_xu$}

Let us find the chaos expansion of
\begin{align*}
\partial_xu(t,x)&=\mathbb E^B\left[\mathcal E(L^x(t))\diamond \bigg\{I_0(u_0'(B_t^x))+u_0(B_t^x)I_1\left(\partial_xL^x(t)\right)\bigg\}\right],
\end{align*}
and notice that in this context, the Wiener integral of order $0$ equals the identity operator, but nonetheless we explicitly write $I_0$ for notational convenience.

By Lemma \ref{lemma: chaos decomposition kernel}, we have
\begin{align*}
	\mathcal E(L^x(t))=\sum_{n=0}^{\infty}\frac{1}{n!}I_n\left(L^{x}(t)^{\otimes n}\right),\quad \mbox{convergent in $L^2\left(\mathbb{P}^W\right)$,}
\end{align*}	
and by the definition of \textit{Wick product}, we see that 
\begin{align*}
	\mathcal E(L^x(t))\diamond\big\{I_0(u_0'(B_t^x))+u_0(B_t^x)I_1\left(\partial_xL^x(t)\right)\big\}=\sum_{n=0}^{\infty}I_n(h_n(t,x)),\quad \mbox{convergent in $(S)^*$,}
\end{align*}
where $h_0(t,x)=u_0'\left(B_t^x\right)$, and
\begin{align*}
\mathscr{S}'(\mathbb R^n)\ni h_n(t,x;\bullet)=u_0(B_t^x)\operatorname{Sym}\left[\left(\frac{L^x(t)^{\otimes (n-1)}}{(n-1)!}\right)\otimes \partial_x L^x(t)\right](\bullet)+ u_0'(B_t^x)\left(\frac{L^x(t)^{\otimes n}}{n!}\right)(\bullet),\ n\geq 1.
\end{align*}

It is known that (e.g. \cite[Chapter 13.3]{K2018})  if  $\Psi(u)=\sum_{n=0}^{\infty} I_n(F_n(u))$ is Bochner integrable on $(M, \sigma(M),m)$, then $F_n$ is Bochner integrable on $(M, \sigma(M),m)$, and it holds that
$$\int_M \Psi(u) m(du)=\sum_{n=0}^{\infty} I_n\left(\int_M F_n(u)m(du)\right).$$ 
In our case, letting $(M,\sigma(M),m)=\left(\widetilde{\Omega},\widetilde{\mathcal{F}},\mathbb{P}^B\right)$, this would imply that
\begin{align*}
	\partial_xu(t,x)&=\sum_{n=0}^{\infty}I_n\left(\mathbb E^B\left[h_n(t,x)\right]\right),
\end{align*}
where $\mathbb{E}^B$ should be understood as a Bochner integral in $\mathscr S'(\mathbb R^n)$.

We can easily check that the first term of $\partial_xu(t,x)$ is $\mathbb{E}^B\left[u_0'\left(B_t^x\right)\right]=\partial_x u_{(\mathbf{0})}(t,x)$. Let's check the general $n$-th terms of $\partial_x u(t,x)$ for $n\geq 1$.
Since $h_n(t,x;\bullet)$ is a symmetric element of $\mathscr{S}'(\mathbb R^n)$, we can expand it with respect to $\{\mathfrak{e}_{\alpha}:\alpha\in\mathcal J_n\}$ as 
\begin{align}\label{eq:h_n expansion}
	h_n(t,x;\bullet)=\sum_{\alpha\in\mathcal J_n} \left\langle h_n(t,x),\mathfrak{e}_{\alpha}\right\rangle_n \mathfrak{e}_{\alpha}(\bullet),\quad\mbox{convergence in $\mathscr{S}'(\mathbb R^n)$},
\end{align}
where $\langle \bullet,\bullet\rangle_n$ is the bilinear product between $\mathscr{S}'(\mathbb R^n)$ and $\mathscr{S}(\mathbb R^n)$, and $$\mathfrak{e}_{\alpha}=\displaystyle\frac{1}{\sqrt{n!\alpha!}}\sum_{\sigma\in \mathcal{P}_n} e_{k_{\sigma(1)}(y_1)}\cdots e_{k_{\sigma(n)}(y_n)}$$ as defined in \eqref{eq:ealpha}.

It is clear by direct calculations that $\langle \operatorname{Sym}f,\operatorname{Sym}g\rangle_n=\langle f, \operatorname{Sym^2}g\rangle_n=\langle f, \operatorname{Sym}g\rangle_n.$  
Then, we have
\begin{align}\label{eq:hnea}
	\left\langle h_n(t,x),\mathfrak{e}_{\alpha}\right\rangle_n
	&=\frac{u_0(B_t^x)}{\sqrt{n!\alpha!}(n-1)!}\left\langle\left[{L^x(t)^{\otimes (n-1)}}\otimes\;\partial_xL^x(t)\right],\sum_{\sigma\in \mathcal P_{n}} \left[e_{k_{\sigma(1)}}\otimes\cdots{\otimes} e_{k_{\sigma(n)}}\right]\;\right\rangle_n\nonumber\\
	&\quad+\frac{u_0'(B_t^x)}{\sqrt{n!\alpha!}n!} \int_{\mathbb R^n}{L^x(t)^{\otimes (n)}}(y_1,...,y_n) \times \sum_{\sigma\in \mathcal P_{n}} \left[e_{k_{\sigma(1)}}\otimes\cdots{\otimes} e_{k_{\sigma(n)}}\right](y_1,...,y_n)d\mathbf{y}\nonumber\\
	&=\frac{u_0(B_t^x)}{\sqrt{n!\alpha!}(n-1)!}\sum_{\sigma\in \mathcal P_{n}}\int_{[0,t]^n}e_{k_{\sigma(1)}}\otimes\cdots{\otimes} e_{k_{\sigma(n)}}'(B_{s_1}^x,\dots,B_{s_n}^x)\;d\mathbf{s}\nonumber\\
	&\quad+\frac{u_0'(B_t^x)}{\sqrt{n!\alpha!}n!}\sum_{\sigma\in \mathcal P_{n}}\int_{[0,t]^n}e_{k_{\sigma(1)}}\otimes\cdots{\otimes} e_{k_{\sigma(n)}}(B_{s_1}^x,\dots,B_{s_n}^x)\;d\mathbf{s},
\end{align}
where in the last expression we used the occupation time formula and the fact that $-\partial_{x}L^x$ equals distributional derivative of the Brownian local time.

Also, taking $\mathbb E^B$ on both sides of \eqref{eq:h_n expansion}, we have
\begin{align}\label{EBhp}
	\mathbb E^B[h_n(t,x;\bullet)]&=\mathbb{E}^B\left[\sum_{\alpha\in\mathcal J_n}\left\langle h_n(t,x),\mathfrak{e}_{\alpha}\right\rangle_n\mathfrak{e}_{\alpha}(\bullet)\right].
\end{align}

At this point, we need the following Lemma to compute \eqref{EBhp}.
\begin{lemma}\cite[Lemma 11.45]{CB2013}
	Let $f:M\to X$ be Bochner integrable on $(M,\sigma(M),m)$ in $X$ and let $Y$ be a Banach space. If $T:X\to Y$ is a bounded operator, then $Tf:M\to Y$ is Bochner integrable on $(M,\sigma(M),m)$ in $Y$ and it holds that
	\begin{align*}
		\int_{M} Tfdm=T\left(\int_{M}fdm\right).
	\end{align*}
\end{lemma}

In our case, we set $\mathscr{S}'_{\mbox{sym}}(\mathbb R^n):=$ the symmetric part of $\mathscr{S}'(\mathbb R^n)$ and $T:\mathscr{S}'(\mathbb R^n)\to \mathscr{S}'_{\mbox{sym}}(\mathbb R^n)$
equals the orthogonal projection on $\mathscr{S}'_{\mbox{sym}}(\mathbb R^n)$, which is clearly a bounded linear operator.
Even though if it is well-known that the space of tempered distributions is not a Banach space, we can think of $\mathscr{S}'(\mathbb R)$ as the inductive limit of a family of Hilbert spaces (e.g \cite[Section 3.2]{K2018} or \cite{PR1991}), and an analogous reasoning extends to the multi-dimensional case; thus the lemma above holds true by letting $Y$ be some of those Hilbert spaces.

Then, \eqref{EBhp} becomes
\begin{align*}
	\mathbb E^B[h_n(t,x;\bullet)]&=\mathbb{E}^B\left[\sum_{\alpha\in\mathcal J_n}\langle h_n(t,x),\mathfrak{e}_{\alpha}\rangle_n \mathfrak{e}_{\alpha}(\bullet)\right]\nonumber
	=\sum_{\alpha\in\mathcal J_n} \left\langle \mathbb{E}^B\left[h_n(t,x)\right],\mathfrak{e}_{\alpha}\right\rangle_n\mathfrak{e}_{\alpha}(\bullet).
\end{align*}

It is known that if a function is Bochner integrable, then its Pettis and Bochner integrals coincide (see for instance the discussion on page 80 of \cite{HP1996}). Therefore, by definition of the Pettis integral, we have 
\begin{align*}
	\left\langle \mathbb{E}^B\left[h_n(t,x)\right],\mathfrak{e}_{\alpha}\right\rangle_n=\mathbb E^B\left[\langle h_n(t,x),\mathfrak{e}_{\alpha}\rangle_n\right].
\end{align*}
Hence, we have
\begin{align*}
	\mathbb E^B[h_n(t,x;\bullet)]=\sum_{\alpha\in\mathcal J_n} \mathbb E^B\left[\langle h_n(t,x),\mathfrak{e}_{\alpha}\rangle_n\right]\mathfrak{e}_{\alpha}(\bullet).
\end{align*}
Next we compute $\mathbb{E}^B\left[\left\langle h_n(t,x),\mathfrak{e}_{\alpha}\right\rangle_n\right]$, and so far from
\eqref{eq:hnea}, we have
\begin{align}\label{eq:EBhnea}
	\mathbb{E}^B\left[\left\langle h_n(t,x),\mathfrak{e}_{\alpha}\right\rangle_n\right]&=\mathbb E^B\bigg[\frac{u_0(B_t^x)}{\sqrt{n!\alpha!}(n-1)!}\int_{[0,t]^n}\sum_{\sigma\in \mathcal P_{n}}e_{k_{\sigma(1)}}\otimes\cdots{\otimes} e_{k_{\sigma(n)}}'(B_{s_1}^x,\dots,B_{s_n}^x)\;d\mathbf{s}\nonumber\\
	&\quad\quad\quad+\frac{u_0'(B_t^x)}{\sqrt{n!\alpha!}n!}\sum_{\sigma\in \mathcal P_{n}}\int_{[0,t]^n}e_{k_{\sigma(1)}}\otimes\cdots{\otimes} e_{k_{\sigma(n)}}(B_{s_1}^x,\dots,B_{s_n}^x)\;d\mathbf{s} \bigg].
\end{align}

To simplify the expression in \eqref{eq:EBhnea}, we first observe
\begin{align*}
	\partial_{x} \left[e_{k_1}(B_{s_1}^x)\cdots e_{k_n}(B_{s_n}^x)\right]=\underbrace{\left[e_{k_1}'(B_{s_1}^x)\cdots e_{k_n}(B_{s_n}^x)\right]+\cdots +\left[e_{k_1}(B_{s_1}^x)\cdots e_{k_n}'(B_{s_n}^x)\right]}_{\mbox{$n$ terms}}.
\end{align*}
Since the Lebesgue measure is invariant under rotations, we see that for any $f:[0,t]^n\to\mathbb R$, it holds that 
\begin{align*}
	\int_{[0,t]^n}f(s_1,\dots,s_n)\;d\mathbf{s}=\int_{[0,t]^n}\operatorname{Sym}f(s_1,\dots,s_n)\;d\mathbf{s}.
\end{align*}
Thus, for any permutation $\sigma$ of $\{1,\dots,n\}$, we have
\begin{align*}
	\int_{[0,t]^n}\frac{u_0(B_t^x)}{(n-1)!}\sum_{\sigma\in \mathcal P_{n}}e_{k_{\sigma(1)}}\otimes\cdots{\otimes} e_{k_{\sigma(n)}}'(B_{s_1}^x,\dots,B_{s_n}^x)\;d\mathbf{s}=\int_{[0,t]^n}u_0(B_t^x)\partial_{x}\left[e_{k_1}(B_{s_1}^x)\cdots e_{k_n}(B_{s_n}^x)\right]\;d\mathbf{s},
\end{align*}
and 
\begin{align*}
\int_{[0,t]^n}\frac{u_0'(B_t^x)}{n!}\sum_{\sigma\in \mathcal P_{n}}e_{k_{\sigma(1)}}\otimes\cdots{\otimes} e_{k_{\sigma(n)}}(B_{s_1}^x,\dots,B_{s_n}^x)\;d\mathbf{s}=\int_{[0,t]^n}u_0'(B_t^x)e_{k_1}(B_{s_1}^x)\cdots e_{k_n}(B_{s_n}^x)\;d\mathbf{s},
\end{align*}
which implies 
\begin{align}\label{eq:EBhnea-0}
\mathbb{E}^B\left[\left\langle h_n(t,x),\mathfrak{e}_{\alpha}\right\rangle_n\right]=\frac{1}{\sqrt{n!\alpha!}}\mathbb{E}^B\left[\int_{\mathbb T^n_{[0,t]}}\partial_{x}\left[u_0(B_t^x)e_{k_1}(B_{s_1}^x)\cdots e_{k_n}(B_{s_n}^x)\right]\;d\mathbf{s}\right],
\end{align}
where again $\mathbb T^n_{[0,t]}=\{0\leq s_1\leq \cdots\leq s_n\leq t\}$.

Furthermore, we notice that
\begin{align*}
&\int_{[0,t]^n}\left[e_{k_1}(B_{s_1}^x)\cdots e_{k_n}(B_{s_n}^x)\right]\;d\mathbf{s}=\int_{[0,t]^n}\operatorname{Sym}\left[e_{k_1}(B_{s_1}^x)\cdots e_{k_n}(B_{s_n}^x)\right]\;d\mathbf{s}\\
&=\int_{[0,t]^n}\frac{1}{n!}\sum_{\sigma\in \mathcal P_{p}}\left[e_{k_1}(B_{s_{\sigma(1)}}^x)\cdots e_{k_p}(B_{s_{\sigma(n)}}^x)\right]\;d\mathbf{s}=\sqrt{\frac{\alpha!}{n!}}\int_{[0,t]^n}\partial_{x}\mathfrak{e}_{\alpha}(B_{s_1}^x,\dots,B_{s_n}^x)\;d\mathbf{s}\\
&=\sqrt{{\alpha!n!}}\int_{\mathbb T^n_{[0,t]}}\partial_{x}\mathfrak{e}_{\alpha}(B_{s_1}^x,\dots,B_{s_n}^x)\;d\mathbf{s},
\end{align*}
and similarly,
\begin{align*}
	&\int_{[0,t]^n}\partial_{x}\left[e_{k_1}(B_{s_1}^x)\cdots e_{k_n}(B_{s_n}^x)\right]\;d\mathbf{s}
	=\sqrt{{\alpha!n!}}\int_{\mathbb T^n_{[0,t]}}\partial_{x}\mathfrak{e}_{\alpha}(B_{s_1}^x,\dots,B_{s_n}^x)\;d\mathbf{s}.
\end{align*}
Therefore, \eqref{eq:EBhnea-0} becomes
\begin{align}\label{eq:EBhnea-1}
	\mathbb{E}^B\left[\left\langle h_n(t,x),\mathfrak{e}_{\alpha}\right\rangle_n\right]&=\mathbb{E}^B\left[\int_{\mathbb T^n_{[0,t]}}\partial_{x}\left[u_0(B_t^x)\mathfrak{e}_{\alpha}(B_{s_1}^x,\dots,B_{s_n}^x)\right]\;d\mathbf{s}\right]\nonumber\\
	&=
	\int_{\mathbb T^n_{[0,t]}}\mathbb{E}^B\left[\partial_{x}\left[u_0(B_t^x)\mathfrak{e}_{\alpha}(B_{s_1}^x,\dots,B_{s_n}^x)\right]\right]\;d\mathbf{s},
\end{align}
where the second equality holds true by Fubini lemma.

By conditioning iteratively on the filtration of $B$ at the sites $s_n,s_{n-1},\dots,s_1$, we can rewrite \eqref{eq:EBhnea-1} as
\begin{align}\label{eq:EBhnea-2}
	\mathbb{E}^B\left[\left\langle h_n(t,x),\mathfrak{e}_{\alpha}\right\rangle_n\right]&=\int_{\mathbb T^n_{[0,t]}}\int_{\mathbb R^n}\partial_{x}\mathfrak{e}_{\alpha}(y_1+x,y_2,\dots,y_n) p(s_n-s_{n-1},y_n-y_{n-1})\times \nonumber\\
	&\quad\quad \qquad\qquad \cdots\times p(s_1,y_1)u_{(\mathbf{0})}(t-s_n,y_n)d\mathbf{y}\; d\mathbf{s}\nonumber\\
	&=\int_{\mathbb T^n_{[0,t]}}\int_{\mathbb R^p}\partial_{y_1}\mathfrak{e}_{\alpha}(y_1+x,y_2,\dots,y_n)p(s_n-s_{n-1},y_n-y_{n-1})\nonumber\\
	&\quad\quad \qquad\qquad  \cdots\times p(s_1,y_1)u_{(\mathbf{0})}(t-s_n,y_n) d\mathbf{y}\; d\mathbf{s}\nonumber\\
	&=-\int_{\mathbb T^n_{[0,t]}}\int_{\mathbb R^n}\mathfrak{e}_{\alpha}(y_1,y_2,\dots,y_n)p(s_n-s_{n-1},y_n-y_{n-1})\nonumber\\
	&\quad\quad \quad\quad  \quad\cdots\times \partial_{y_1}p(s_1,y_1-x)u_{(\mathbf{0})}(t-s_n,y_n)d\mathbf{y}\; d\mathbf{s}.
\end{align}

Noticing that $-\partial_y p(s,y-x)=\partial_x p(s,x-y)$ and letting $r_i=t-s_{n+1-i}$ for $i\in\{1,\dots,n\}$, the equation \eqref{eq:EBhnea-2} becomes
\begin{align*}
	\int_{\mathbb T^n_{[0,t]}}\int_{\mathbb R^n}\mathfrak{e}_{\alpha}(y_1,y_2,\dots,y_n)p(r_2-r_1,y_n-y_{n-1})\times \cdots\times \partial_{x} p(t-r_n,x-y_1)u_{(\mathbf{0})}(r_1,y_n)d\mathbf{y}\; d\mathbf{r},
\end{align*} 
which yields, after relabeling the $y$'s,
\begin{align*}
	\mathbb{E}^B\left[\left\langle h_n(t,x),\mathfrak{e}_{\alpha}\right\rangle_n\right]=\int_{\mathbb T^n_{[0,t]}}\int_{\mathbb R^n}&\partial_{x} p(t-r_n,x-y_n) p(r_n-r_{n-1},y_{n}-y_{n-1})\times \cdots\times  p(r_2-r_1,y_2-y_1)\\
	&\cdots\times \mathfrak{e}_{\alpha}(y_1,y_2,\dots,y_n)u_{(\mathbf{0})}(r_1,y_1) d\mathbf{y}\; d\mathbf{r}.
\end{align*}

Using the definition of $\mathfrak{e}_{\alpha}$, it equals
\begin{align*}
	\frac{1}{\sqrt{\alpha!}}\frac{1}{\sqrt{n!}}&\sum_{\sigma\in \mathcal P_{n}}\int_{\mathbb T^n_{[0,t]}}\int_{\mathbb R^n}\partial_{x} p(t-r_n,x-y_n) p(r_n-r_{n-1},y_{n}-y_{n-1})\times \cdots\times  p(r_2-r_1,y_2-y_1)\\
	&\cdots\times e_{k_{\sigma(1)}}(y_1)\times\cdots\times e_{k_{\sigma(n)}}(y_n)u_{(\mathbf{0})}(r_1,y_1)d\mathbf{y}\; d\mathbf{r}=:\frac{1}{\sqrt{n!}}\mathfrak{K}_{\alpha}(t,x),\quad |\alpha|=n\geq 1.
\end{align*}

Putting all together, we obtain
\begin{align*}
	\mathbb E^B\left[h_n(t,x;\bullet)\right]=\sum_{\alpha\in\mathcal J_n}\frac{1}{\sqrt{n!}}\mathfrak{K}_{\alpha}(t,x)\mathfrak{e}_{\alpha}(\bullet)
\end{align*}
and
\begin{align*}
I_n\left(\mathbb E^B\left[h_n(t,x)\right]\right)&=\sum_{\alpha\in\mathcal J_n}\mathfrak{K}_{\alpha}(t,x)I_n\left(\frac{\mathfrak{e}_{\alpha}}{\sqrt{n!}}\right)=\sum_{\alpha\in\mathcal J_n}\mathfrak{K}_{\alpha}(t,x)\xi_{\alpha} \quad \mbox{by \eqref{eq:Inxia}.}
\end{align*}

Finally, we have the chaos expansion of $\partial_x u(t,x)$ as follows:
\begin{align}\label{eq:uxChaos}
	\partial_{x} u(t,x)=\sum_{n=0}^{\infty}I_n\left(\mathbb E^B\left[h_n(t,x)\right]\right)=\sum_{n=0}^{\infty}\sum_{\alpha\in\mathcal J_n}\mathfrak{K}_{\alpha}(t,x)\xi_{\alpha}=\sum_{\alpha\in\mathcal J}\mathfrak{K}_{\alpha}(t,x)\xi_{\alpha},
\end{align}
where $\mathfrak{K}_{(\mathbf{0})}(t,x)=\partial_x u_{(\mathbf{0})}(t,x)$, and for $|\alpha|=n\geq 1$,
\begin{align*}
	\mathfrak{K}_{\alpha}(t,x)=\frac{1}{\sqrt{\alpha!}}\sum_{\sigma\in \mathcal P_{n}}\int_{\mathbb T^n_{[0,t]}}\int_{\mathbb R^n}&\partial_{x} p(t-r_n,x-y_n)p(r_n-r_{n-1},y_{n}-y_{n-1})\times \cdots\times  p(r_2-r_1,y_2-y_1)\\
	&\cdots\times e_{k_{\sigma(1)}}(y_1)\times\cdots\times e_{k_{\sigma(n)}}(y_n)u_{(\mathbf{0})}(r_1,y_1)d\mathbf{y}\; d\mathbf{r}.
\end{align*}

\begin{remark}
	Using the Feynman-Kac representation of $u(t,x)$, it was possible to compute the weak derivative of $u(t,x)$ with respect to $x$ in $(S)^*$ as follows: $$\partial_x u(t,x)=\sum_{\alpha\in\mathcal J}\mathfrak{K}_{\alpha}(t,x)\xi_{\alpha},$$
	where $\mathfrak{K}_{(\mathbf{0})}(t,x)=\partial_x u_{(\mathbf{0})}(t,x)$, and for $|\alpha|=n\geq 1$,
	\begin{align*}
	\mathfrak{K}_{\alpha}(t,x)=\frac{1}{\sqrt{\alpha!}}\sum_{\sigma\in \mathcal P_{n}}\int_{\mathbb T^n_{[0,t]}}\int_{\mathbb R^n}&\partial_{x} p(t-r_n,x-y_n)p(r_n-r_{n-1},y_{n}-y_{n-1})\times \cdots\times  p(r_2-r_1,y_2-y_1)\\
	&\cdots\times e_{k_{\sigma(1)}}(y_1)\times\cdots\times e_{k_{\sigma(n)}}(y_n)u_{(\mathbf{0})}(r_1,y_1)d\mathbf{y}\; d\mathbf{r}.
	\end{align*}
	In fact, the weak spatial derivative is the usual spatial derivative in $L^2\left(\mathbb{P}^W\right)$ sense (and hence $\mathbb{P}^W$ almost sure sense). One can verify this assertion by computing $\partial_x u(t,x)$ using
	the chaos expansion of $u(t,x)$ directly and by the uniqueness of mild solution.
\end{remark}

Before we show $\partial_xu(t,x)\in \mathcal{G}$ for each $t>0$ and $x\in \mathbb{R}$, we first need the following lemma:
\begin{lemma}\label{lem:Ke->K}
	Assume that $u_0\in C_b^1(\mathbb{R})$. Then,
	for each $\alpha\in \mathcal{J}$, $t>0$ and $x\in \mathbb{R}$, $\mathfrak{K}_{\alpha}(t,x)$ is well-defined, and moreover, for $|\alpha|\geq 1$,
	$$\lim_{\varepsilon\to 0^+} \mathfrak{K}_{\alpha}^{\epsilon}(t,x)=\mathfrak{K}_{\alpha}(t,x),$$
	where
	\begin{align*}
	\mathfrak{K}_{\alpha}^{\epsilon}(t,x):=\frac{1}{\sqrt{\alpha!}}\sum_{\sigma\in \mathcal P_{n}}\int_{\mathbb T^n_{[0,t-\epsilon]}}\int_{\mathbb R^n}&\partial_{x} p(t-r_n,x-y_n)p(r_n-r_{n-1},y_{n}-y_{n-1})\times \cdots\times  p(r_2-r_1,y_2-y_1)\\
	&\cdots\times e_{k_{\sigma(1)}}(y_1)\times\cdots\times e_{k_{\sigma(n)}}(y_n)u_{(\mathbf{0})}(r_1,y_1)d\mathbf{y}\; d\mathbf{r}\qquad \mbox{for $\epsilon>0$},
	\end{align*}
	with $\mathbb T^n_{[0,t-\epsilon]}:=\{0\leq s_1\leq\cdots\leq s_n\leq t-\epsilon\}$.
\end{lemma}
\begin{proof}
	We will decompose $\mathfrak{K}_{\alpha}(t,x)$ as a finite sum of well-defined terms.	Without loss of generality, we let $|\alpha|=n\geq 1$.
	Notice that
	\begin{align*}
	\mathfrak{K}_{\alpha}(t,x)&=\frac{1}{\sqrt{\alpha!}}\sum_{\sigma\in \mathcal P_{n}}\int_{\mathbb T^n_{[0,t]}}\int_{\mathbb R^n}\partial_{x} p(t-r_n,x-y_n)p(r_n-r_{n-1},y_{n}-y_{n-1})\times \cdots\times  p(r_2-r_1,y_2-y_1)\\
	&\qquad \qquad \qquad \qquad \quad \times e_{k_{\sigma(1)}}(y_1)\times\cdots\times e_{k_{\sigma(n)}}(y_n)u_{(\mathbf{0})}(r_1,y_1)d\mathbf{y}\; d\mathbf{r}\\
	&=-\frac{1}{\sqrt{\alpha!}}\sum_{\sigma\in \mathcal P_{n}}\int_{\mathbb T^n_{[0,t]}}\int_{\mathbb R^n}\partial_{y_n} p(t-r_n,x-y_n)p(r_n-r_{n-1},y_{n}-y_{n-1})\times \cdots\times  p(r_2-r_1,y_2-y_1)\\
	&\qquad \qquad \qquad \qquad \quad \times e_{k_{\sigma(1)}}(y_1)\times\cdots\times e_{k_{\sigma(n)}}(y_n)u_{(\mathbf{0})}(r_1,y_1)d\mathbf{y}\; d\mathbf{r}\\
	&=\frac{1}{\sqrt{\alpha!}}\sum_{\sigma\in \mathcal P_{n}}\int_{\mathbb T^n_{[0,t]}}\int_{\mathbb R^n} p(t-r_n,x-y_n)\partial_{y_{n}}\left(p(r_n-r_{n-1},y_{n}-y_{n-1})e_{k_{\sigma(n)}}(y_n)\right)\times \cdots  \\
	&\qquad \qquad \qquad \qquad \quad \times p(r_2-r_1,y_2-y_1) e_{k_{\sigma(1)}}(y_1)\times\cdots\times e_{k_{\sigma(n-1)}}(y_{n-1}) u_{(\mathbf{0})}(r_1,y_1)d\mathbf{y}\; d\mathbf{r}\\
	&=\frac{1}{\sqrt{\alpha!}}\sum_{\sigma\in \mathcal P_{n}}\int_{\mathbb T^n_{[0,t]}}\int_{\mathbb R^n} p(t-r_n,x-y_n)\partial_{y_{n}}p(r_n-r_{n-1},y_{n}-y_{n-1})e_{k_{\sigma(n)}}(y_n)\times \cdots  \\
	&\qquad \qquad \qquad \qquad \quad \times p(r_2-r_1,y_2-y_1) e_{k_{\sigma(1)}}(y_1)\times\cdots\times e_{k_{\sigma(n-1)}}(y_{n-1}) u_{(\mathbf{0})}(r_1,y_1)d\mathbf{y}\; d\mathbf{r}\\
	&\quad +\frac{1}{\sqrt{\alpha!}}\sum_{\sigma\in \mathcal P_{n}}\int_{\mathbb T^n_{[0,t]}}\int_{\mathbb R^n} p(t-r_n,x-y_n)p(r_n-r_{n-1},y_{n}-y_{n-1})e'_{k_{\sigma(n)}}(y_n)\times \cdots  \\
	&\qquad \qquad \qquad \qquad \quad \times p(r_2-r_1,y_2-y_1) e_{k_{\sigma(1)}}(y_1)\times\cdots\times e_{k_{\sigma(n-1)}}(y_{n-1}) u_{(\mathbf{0})}(r_1,y_1)d\mathbf{y}\; d\mathbf{r}\\
	&=:a_1+b_1.
	\end{align*}
	We note that $b_1$ is well-defined since $u_0\in L^{\infty}(\mathbb{R})$, and for $a_1$, we do a similar step as above:
	\begin{align*}
	a_1&=-\frac{1}{\sqrt{\alpha!}}\sum_{\sigma\in \mathcal P_{n}}\int_{\mathbb T^n_{[0,t]}}\int_{\mathbb R^n} p(t-r_n,x-y_n)\partial_{y_{n-1}}p(r_n-r_{n-1},y_{n}-y_{n-1})e_{k_{\sigma(n)}}(y_n)\times \cdots  \\
	&\qquad \qquad \qquad \qquad \quad \times p(r_2-r_1,y_2-y_1) e_{k_{\sigma(1)}}(y_1)\times\cdots\times e_{k_{\sigma(n-1)}}(y_{n-1}) u_{(\mathbf{0})}(r_1,y_1)d\mathbf{y}\; d\mathbf{r}\\
	&=\frac{1}{\sqrt{\alpha!}}\sum_{\sigma\in \mathcal P_{n}}\int_{\mathbb T^n_{[0,t]}}\int_{\mathbb R^n} p(t-r_n,x-y_n)p(r_n-r_{n-1},y_{n}-y_{n-1})e_{k_{\sigma(n)}}(y_n)\\
	&\qquad \qquad \qquad \qquad \quad\times \partial_{y_{n-1}}\left(p(r_{n-1}-r_{n-2},y_{n-1}-y_{n-2})e_{k_{\sigma(n-1)}}(y_{n-1})\right)\times \cdots  \\
	&\qquad \qquad \qquad \qquad \quad \times p(r_2-r_1,y_2-y_1) e_{k_{\sigma(1)}}(y_1)\times\cdots\times e_{k_{\sigma(n-2)}}(y_{n-2}) u_{(\mathbf{0})}(r_1,y_1)d\mathbf{y}\; d\mathbf{r}\\
	&=\frac{1}{\sqrt{\alpha!}}\sum_{\sigma\in \mathcal P_{n}}\int_{\mathbb T^n_{[0,t]}}\int_{\mathbb R^n} p(t-r_n,x-y_n)p(r_n-r_{n-1},y_{n}-y_{n-1})e_{k_{\sigma(n)}}(y_n)\\
	&\qquad \qquad \qquad \qquad \quad\times \partial_{y_{n-1}}p(r_{n-1}-r_{n-2},y_{n-1}-y_{n-2})e_{k_{\sigma(n-1)}}(y_{n-1})\times \cdots  \\
	&\qquad \qquad \qquad \qquad \quad \times p(r_2-r_1,y_2-y_1) e_{k_{\sigma(1)}}(y_1)\times\cdots\times e_{k_{\sigma(n-2)}}(y_{n-2}) u_{(\mathbf{0})}(r_1,y_1)d\mathbf{y}\; d\mathbf{r}\\
	&\quad +\frac{1}{\sqrt{\alpha!}}\sum_{\sigma\in \mathcal P_{n}}\int_{\mathbb T^n_{[0,t]}}\int_{\mathbb R^n} p(t-r_n,x-y_n)p(r_n-r_{n-1},y_{n}-y_{n-1})e_{k_{\sigma(n)}}(y_n)\\
	&\qquad \qquad \qquad \qquad \quad\times p(r_{n-1}-r_{n-2},y_{n-1}-y_{n-2})e'_{k_{\sigma(n-1)}}(y_{n-1})\times \cdots  \\
	&\qquad \qquad \qquad \qquad \quad \times p(r_2-r_1,y_2-y_1) e_{k_{\sigma(1)}}(y_1)\times\cdots\times e_{k_{\sigma(n-2)}}(y_{n-2}) u_{(\mathbf{0})}(r_1,y_1)d\mathbf{y}\; d\mathbf{r}\\
	&=:a_2+b_2.
	\end{align*}
	Again, $b_2$ is well-defined. By iterating this process, we can get 
	$$\mathfrak{K}_{\alpha}(t,x)=a_{n-1}+\sum_{i=1}^{n-1}b_i,$$
	where $\displaystyle\sum_{i=1}^{n-1}b_i$ is well-defined and
	\begin{align*}
	a_{n-1}&=\frac{1}{\sqrt{\alpha!}}\sum_{\sigma\in \mathcal P_{n}}\int_{\mathbb T^n_{[0,t]}}\int_{\mathbb R^n} p(t-r_n,x-y_n)p(r_n-r_{n-1},y_{n}-y_{n-1})\times \cdots\times  p(r_2-r_1,y_2-y_1)\\
	&\times e_{k_{\sigma(1)}}(y_1)\times\cdots\times e_{k_{\sigma(n)}}(y_n)\partial_{y_{1}}u_{(\mathbf{0})}(r_1,y_1)d\mathbf{y}\; d\mathbf{r}.
	\end{align*}
	Since $u_{(\mathbf{0})}(r_1,y_1)d\mathbf{y}=\displaystyle\int_{\mathbb{R}} p(r_1,y_1-y_0)u_0(y_0)dy_0$, $a_{n-1}$ becomes
	\begin{align*}
	a_{n-1}&=\frac{1}{\sqrt{\alpha!}}\sum_{\sigma\in \mathcal P_{n}}\int_{\mathbb T^n_{[0,t]}}\int_{\mathbb R^{n+1}} p(t-r_n,x-y_n)p(r_n-r_{n-1},y_{n}-y_{n-1})\times \cdots\times  p(r_2-r_1,y_2-y_1)\\
	&\times e_{k_{\sigma(1)}}(y_1)\times\cdots\times e_{k_{\sigma(n)}}(y_n)p(r_1,y_1-y_0)u_0'(y_0)dy_0d\mathbf{y}\; d\mathbf{r},
	\end{align*}
	which is clearly well-defined since $u_0'\in L^{\infty}(\mathbb{R})$.
	
	Moreover, one can show that $\displaystyle\lim_{\varepsilon\to 0^+} \mathfrak{K}_{\alpha}^{\epsilon}(t,x)=\mathfrak{K}_{\alpha}(t,x)$ easily by considering the same argument as $\mathfrak{K}_{\alpha}(t,x)$ for $\mathfrak{K}_{\alpha}^{\epsilon}(t,x)$.
\end{proof}

\begin{theorem}\label{thm:uxL2}
	If $u_0\in C_b^1(\mathbb{R})$, for each $t>0$ and $x\in \mathbb{R}$,
	$$\partial_x u(t,x)\in \mathcal G.$$
\end{theorem}
\begin{proof}
	From \eqref{eq:uxChaos}, we have
	$$\mathbb{E}|\partial_xu(t,x)|^2=|\mathfrak{K}_{(\mathbf{0})}(t,x)|^2+\sum_{n=1}^{\infty} \sum_{\alpha\in \mathcal{J}_n} |\mathfrak{K}_{\alpha}(t,x)|^2.$$
	We note that $\mathfrak{K}_{(\mathbf{0})}(t,x)=\partial_x u_{(\mathbf{0})}(t,x)<\infty$ for each $(t,x)\in [0,T]\times \mathbb{R}$.
	For $|\alpha|=n\geq 1$,
	we use $\mathfrak{K}_{\alpha}^{\epsilon}$:
	For $\epsilon>0$, we notice that (by Fubini lemma)
	\begin{align*}
	\mathfrak{K}_{\alpha}^{\epsilon}(t,x)=\sqrt{n!}\int_{\mathbb R^n}\int_{\mathbb T^n_{[0,t-\epsilon]}}&\partial_{x} p(t-r_n,x-y_n) p(r_n-r_{n-1},y_{n}-y_{n-1})\times \cdots\times p(r_2-r_1,y_2-y_1)\\
	&\cdots\times \mathfrak{e}_{\alpha}(y_1,\dots,y_n)u_{(\mathbf{0})}(r_1,y_1)\; d\mathbf{r}d\mathbf{y}.
	\end{align*}
	Using Bessel's inequality, we have that 
	\begin{align*}
	&\sum_{\alpha\in\mathcal J_n} |\mathfrak{K}_{\alpha}^{\epsilon}(t,x)|^2\leq n!\int_{\mathbb R^n}\bigg(\int_{\mathbb T^n_{[0,t-\epsilon]}}\partial_{x}p(t-r_n,x-y_n) p(r_n-r_{n-1},y_{n}-y_{n-1})\times \\
	&\qquad \qquad \qquad \qquad \cdots \times p(r_2-r_1,y_2-y_1)u_{(\mathbf{0})}(r_1,y_1)d\mathbf{r}\bigg)^2d\mathbf{y}\\
	&\leq n!\|u_0\|_{\infty}^2\int_{\mathbb R^n}\int_{\mathbb T^n_{[0,t-\epsilon]}}\int_{\mathbb T^n_{[0,t-\epsilon]}}\partial_{x} p(t-r_n,x-y_n)\partial_{x} p(t-s_n,x-y_n)p(r_n-r_{n-1},y_{n}-y_{n-1})\\  &\quad \times p(s_n-s_{n-1},y_{n}-y_{n-1})
	\cdots  p(r_2-r_1,y_2-y_1)p(s_2-s_1,y_2-y_1)\; d\mathbf{r}d\mathbf{s} d\mathbf{y}.
	\end{align*}
	Again using Fubini lemma and the semigroup property of $p$, we can write the last expression as
	\begin{align*}
	n!\|u_0\|_{\infty}^2(2\pi)^{-(n-1)/2}&\int_{\mathbb T^n_{[0,t-\epsilon]}}\int_{\mathbb T^n_{[0,t-\epsilon]}}\prod_{k=1}^{n-1}(s_{k+1}+r_{k+1}-s_k-r_k)^{-1/2}\\
	&\quad\cdots\times \int_{\mathbb R}\partial_{x} p(t-r_n,x-y_n)\partial_{x} p(t-s_n,x-y_n)dy_n\; d\mathbf{r}d\mathbf{s}.
	\end{align*}

	Since
	\begin{align*}
	&\int_{\mathbb R}\partial_{x} p(t-r_n,x-y_n)\partial_{x} p(t-s_n,x-y_n)dy_n
	=(2\pi)^{-1/2}(2t-s_n-r_n)^{-3/2},
	\end{align*}
and $(a+b)^{-1/2}\leq 2^{-1/2}a^{-1/4}b^{-1/4}$ for $a,b>0$, we can finally see that
	\begin{align*}
	\sum_{\alpha\in\mathcal J_n} |\mathfrak{K}_{\alpha}^{\epsilon}(t,x)|^2&\leq n!\|u_0\|_{\infty}^2(2\pi)^{-n/2}\left(\int_{\mathbb T^n_{[0,t-\epsilon]}}(t-s_n)^{-3/4}\prod_{k=1}^{n-1}(s_{k+1}-s_k)^{-1/4}d\mathbf{s}\right)^2\\
	&\leq n!\|u_0\|_{\infty}^2(2\pi)^{-n/2}\left(\int_{\mathbb T^n_{[0,t]}}(t-s_n)^{-3/4}\prod_{k=1}^{n-1}(s_{k+1}-s_k)^{-1/4}d\mathbf{s}\right)^2\\
	&\leq \|u_0\|_{\infty}^2C^nt^{3n-4}n^{-n/2}\quad \mbox{for some $C>0$,}
	\end{align*}
	where the last inequality follows from \cite[equation (4.10)]{KL2017}. 
	
	For each $t>0$ and $x\in \mathbb{R}$,
	the convergence is uniform in $\epsilon$ and $\mathfrak{K}_{\alpha}^{\epsilon}(t,x)\to \mathfrak{K}_{\alpha}(t,x)$ as $\epsilon\to 0$ by Lemma \ref{lem:Ke->K}, 
	\begin{align*}
	\sum_{\alpha\in\mathcal J_n} |\mathfrak{K}_{\alpha}(t,x)|^2\leq \|u_0\|_{\infty}^2C^nt^{\frac{3n}{4}-1}n^{-n/2}\quad \mbox{for some constant $C>0$.}
	\end{align*}
	Since $C^nt^{\frac{3n}{4}-1}n^{-n/2}e^{2\lambda n}$ is summable in $n$ for any $\lambda\in\mathbb R$, the conclusion follows from Definition \ref{def: space G}.
\end{proof}

\begin{remark}
	We have the following Feynman-Kac type formula for the spatial derivative of $u$:	
	\begin{align*}
	\partial_xu(t,x)&=\mathbb E^B\left[\mathcal E(L^x(t))\diamond \bigg\{u_0'(B_t^x)+u_0(B_t^x)I_1\left(\partial_xL^x(t)\right)\bigg\}\right],
	\end{align*}
	where $\mathbb{E}^B$ must be understood as a Bochner integral in $(S)^*$.
	Notice that the integrand of $\mathbb{E}^B$ on the right-hand side is in fact a Hida distribution. However, after taking the expectation $\mathbb{E}^B$, which is interpreted as a Bochner integral in $(S)^*$, we end up with a regular random variable  in $\mathcal{G}$. This means that the \textit{white noise integral} (or the Bochner integral in $(S)^*$) has a regularizing effect. 
\end{remark}

\begin{theorem}\label{thm:uxspace}
	Let $0<\varepsilon<1/2$ be arbitrary and assume that $u_0\in C^{3/2}(\mathbb{R})$. Then, for each $t>0$,
	$$\partial_x u(t,\bullet)\in C^{1/2-\varepsilon}(\mathbb{R}).$$
\end{theorem}
\begin{proof}
Let $p>1$, $t>0$ and $x\in \mathbb{R}$. By \cite[Proposition 2.1]{KL2017}, we have
\begin{align}\label{eq:EpE2}
\left(\mathbb{E}\left|\partial_xu(t,x+h)-\partial_xu(t,x)\right|^p\right)^{1/p}=\sum_{n=0}^{\infty} (p-1)^{n/2}\left(\sum_{\alpha\in \mathcal{J}_n} |\mathfrak{K}_{\alpha}(t,x+h)-\mathfrak{K}_{\alpha}(t,x)|^2\right)^{1/2}.
\end{align}
For $|\alpha|=0$, by Lemma \ref{lem:holder}, we have
$$\mathfrak{K}_{(\mathbf{0})}(t,\bullet)=\partial_x u_{(\mathbf{0})}(t,\bullet)\in C^{1/2}(\mathbb{R}).$$
For $|\alpha|=n\geq 1$,
similarly to the proof of Theorem \ref{thm:uxL2}, we can get
\begin{align*}
&\sum_{\alpha\in \mathcal{J}_n} \left|\mathfrak{K}_{\alpha}^{\epsilon}(t,x+h)-\mathfrak{K}_{\alpha}^{\epsilon}(t,x)\right|^2\\
&\leq n!\int_{\mathbb R^n}\bigg(\int_{\mathbb T^n_{[0,t-\epsilon]}}\left(\partial_{x}p(t-r_n,x+h-y_n)-\partial_{x}p(t-r_n,x-y_n)\right) p(r_n-r_{n-1},y_{n}-y_{n-1})\times \\
&\qquad \qquad \qquad \qquad \cdots \times p(r_2-r_1,y_2-y_1)u_{(\mathbf{0})}(r_1,y_1)d\mathbf{r}\bigg)^2d\mathbf{y}\\
&\leq 	
n!\|u_0\|_{\infty}^2(2\pi)^{-(n-1)/2}\int_{\mathbb T^n_{[0,t-\epsilon]}}\int_{\mathbb T^n_{[0,t-\epsilon]}}\prod_{k=1}^{n-1}(s_{k+1}+r_{k+1}-s_k-r_k)^{-1/2}\\
&\quad\times \int_{\mathbb R}\left(\partial_{x} p(t-r_n,x+h-y_n)-\partial_{x} p(t-r_n,x-y_n)\right)\\
&\qquad\qquad\qquad\qquad\qquad\qquad\quad \cdots\times\left(\partial_{x} p(t-s_n,x+h-y_n)-\partial_{x} p(t-s_n,x-y_n)\right)dy_n\; d\mathbf{r}d\mathbf{s}.
\end{align*}
We next compute
\begin{align*}
\int_{\mathbb R}\partial_{x}p(t_1,x_1-z) \partial_{x} p(t_2,x_1-z)dz&=\frac{1}{2\pi t_1^{3/2}t_2^{3/2}}\int_{\mathbb{R}} (x_1-z)(x_2-z)e^{-\frac{(x_1-z)^2}{2t_1}-\frac{(x_2-z)^2}{2t_2}}dz\\
&=\frac{1}{2\pi t_1^{3/2}t_2^{3/2}}\int_{\mathbb{R}} z\left(z-(x_1-x_2)\right)e^{-\frac{z^2}{2t_1}-\frac{(z-(x_1-x_2))^2}{2t_2}}dz\\
&=\frac{e^{-\frac{(x_1-x_2)^2}{2(t_1+t_2)}}}{2\pi t_1^{3/2}t_2^{3/2}}\int_{\mathbb{R}} (z^2-(x_1-x_2)z)e^{-\frac{(t_1+t_2)}{2t_1t_2}\left(z-\frac{(x_1-x_2)t_1}{t_1+t_2}\right)^2}dz\\
&=\frac{1}{\sqrt{2\pi}} e^{-\frac{(x_1-x_2)^2}{2(t_1+t_2)}} (t_1+t_2)^{-3/2}\left(1-\frac{(x_1-x_2)^2}{(t_1+t_2)}\right).
\end{align*}
The last equality can be verified using the mean and variance of a normal distribution $N\left(\frac{(x-y)t_1}{t_1+t_2},\frac{t_1t_2}{t_1+t_2}\right)$.
Then, we can easily check, using the fact $1-e^{-z}\leq z^{\gamma}$ for any $z\geq 0$ and $0<\gamma\leq 1$,
\begin{align}\label{eq:pxpxh}
&\int_{\mathbb R}\left(\partial_{x} p(t-r_n,x+h-y_n)-\partial_{x} p(t-r_n,x-y_n)\right)\nonumber\\
&\qquad\qquad\qquad\qquad\qquad\qquad\quad \times\left(\partial_{x} p(t-s_n,x+h-y_n)-\partial_{x} p(t-s_n,x-y_n)\right)dy_n\nonumber\\
&=\sqrt{\frac{2}{\pi}}(2t-s-r)^{-3/2}\left(1-e^{-\frac{h^2}{2(2t-s-r)}}+\frac{h^2e^{-\frac{h^2}{2(2t-s-r)}}}{2t-s-r}\right)\leq Ch^{2\gamma} (2t-s-r)^{-3/2-\gamma}.
\end{align}
At this point, we restrict $0<\gamma<1/2$ so that \eqref{eq:pxpxh} is integrable both in $s$ and $r$ variables near $t$.

This leads to
\begin{align*}
&\sum_{\alpha\in \mathcal{J}_n} \left|\mathfrak{K}_{\alpha}^{\epsilon}(t,x+h)-\mathfrak{K}_{\alpha}^{\epsilon}(t,x)\right|^2\leq h^{2\gamma} \|u_0\|_{\infty}^2C^n(\gamma)t^{\frac{3n}{4}-1-\frac{\gamma}{2}}n^{-n/2},
\end{align*}
with $0<\gamma<1/2$ for some constant $C(\gamma)>0$ depending only on $\gamma$. 

After taking $\epsilon\to 0$, the desired result follows from \eqref{eq:EpE2} and the Kolmogorov continuity theorem.
\end{proof}

\begin{remark}
	Under the same initial condition in Theorem \ref{thm:uxspace}, we can achieve the optimal temporal regularity of $\partial_x u$ in a similar manner, i.e., $\partial_x u(\bullet,x)\in C^{1/4-\varepsilon}\left([\varepsilon_0,T]\right)$ for every $x\in \mathbb{R}$, $0<\varepsilon_0<T$, and $0<\varepsilon<1/4$.
\end{remark}



\bibliographystyle{alpha}
\def\cprime{$'$}

\end{document}